\documentclass[11pt]{amsart}
\usepackage{hyperref}

\usepackage[margin=3cm]{geometry}

\usepackage{longtable}
\usepackage{lscape}

\newcommand{\sltwo}{\mathfrak{sl}(2)}
\newcommand{\mone}{\underline{1}}
\newcommand{\mtwo}{\underline{2}}
\newcommand{\mthree}{\underline{3}}
\newcommand{\Tbcd}{T(b,c,d)}
\newcommand{\Tbgd}{T(\beta,\gamma,\delta)}
\newcommand{\Ttil}{\widetilde{T}(b,\frac{1}{b},0)}

\newcommand{\xpm}{X_+ X_-}

\newcommand {\fsl}    {{\mathfrak{sl}}}
\newcommand{\mat}[1]{\text{\small\arraycolsep=4pt $\begin{pmatrix}#1\end{pmatrix}$}}
\newcommand\defis{\nobreak\hskip0pt\hbox{-}\nobreak\hskip0pt}
\newcommand\Tire{\ifdim\lastskip>0pt\unskip\nobreak\hskip1.75pt plus.75pt minus.5pt\hbox{---}\hskip1.75pt plus.75pt minus.5pt\else \hbox{---}\hskip2pt plus1pt
minus.5pt\fi}

\usepackage[all]{xy}

\newtheorem{thm}{Theorem}

\newtheorem{lem}[thm]{Lemma}
\theoremstyle{definition}

\theoremstyle{definition}

\theoremstyle{remark}
\newtheorem{rmk}[thm]{Remark}

\def\c+{\rlap{\ \raisebox{.2ex}{\textnormal{\scriptsize+}}}\supset}

\begin{document}

\begin{abstract}
We completely describe the decompositions (into indecomposable
submodules) of the tensor products of irreducible $\sltwo$-modules
in characteristic 3. The answer resembles analogous decompositions
for the Lie superalgebra $\fsl(1|1)$.
\end{abstract}

\title[Tensor products of irreducible
  $\mathfrak{sl}(2)$-modules]{Decompositions of the tensor products of
  irreducible $\sltwo$-modules in characteristic 3}
\author{Brian Clarke}
\date{September 22, 2008}
\thanks{The author would like to thank Dimitry Leites for raising this
  problem and for helpful guidance along the way, Pavel Grozman for
  assistance with his excellent \texttt{SuperLie} package, as well as
  the MPIMiS and IMPRS for financial support and providing a creative atmosphere.}
\address{International Max Planck Research School, Max Planck
  Institute for Mathematics in the Sciences, Inselstra\ss e 22, 04103
  Leipzig, Germany; clarke@mis.mpg.de}
\maketitle

\section{Introduction}

Texts devoted to representations of Lie algebras in characteristic
$p>0$ are often prefaced by the disclaimer that the spaces
considered are of dimension less than $p$. To the best of the
author's knowledge, this restriction is always imposed on tensor
products of irreducible modules when studying the analog of
Klebsch-Gordon decompositions. That is, if $V$ and $W$ are two
irreducible modules over a simple Lie algebra $\mathfrak{g}$ and
one wishes to decompose $V \otimes W$ into indecomposable
submodules, $\dim V \otimes W$ is always restricted to be less
than $p$. In this paper, I remove this restriction and present a
complete investigation of the decomposition of $V \otimes W$ for
the case $\mathfrak{g} = \sltwo$ and $p=3$ for any irreducible
$\sltwo$-modules $V$ and $W$.

In $\sltwo$, we consider the natural basis
\[
 X_-=\mat{ 0&0\\1&0 },\quad
 H=\mat{ 1&0\\0&-1 },\quad
 X_+=\mat{ 0&1\\0&0 }.
\]
With this, the structure constants are derived from the relations
\[
[X_+,X_-]=H,\quad [H,X_{\pm}]=\pm2X_{\pm}.
\]

Let $k$ be an algebraically closed field of characteristic 3.
Irreducible $\sltwo$-modules in characteristic $p>2$ were
completely described by Rudakov and Shafarevich in
\cite{rudakov-shafarevich}. These modules are all of dimension $D
\leq p$, and in the cases $D<p$, there is no difference from the
case of characteristic zero (cf.~\cite{fulton-harris}). For $p=3$ and $D<3$,
these are only the modules denoted by $\mone$,
\begin{equation}
\xymatrix{
0 & V_0 \ar[l]^-{X_+} \ar[r]^-{X_-} & 0,
}
\end{equation}
and $\mtwo$,
\begin{equation}
\xymatrix{
0 & V_1 \ar[l]^-{X_+} \ar@<.5ex>[r]^-{X_-} & V_{-1} \ar@<.5ex>[l]^-{X_+}
\ar[r]^-{X_-}  & 0.
}
\end{equation}
In these diagrams, $V_{\rho}$ denotes the 1\defis dimensional
\emph{weight space} of eigenvectors of $H$ with eigenvalue $\rho$.
The arrows indicate the action of the operators  in the sub- or
superscript.

\begin{rmk}
More generally, over a field of prime characteristic $p$, we always
have the irreducible $\sltwo$ modules $\underline{N}$ for $N \in
\{1,\dots,p\}$, given diagrammatically by
\begin{equation}\label{eq:N}
\xymatrix{
0 & V_{N-1} \ar[l]^-{X_+} \ar@<.5ex>[r]^-{X_-} & V_{N-3} \ar@<.5ex>[l]^-{X_+}
\ar@<.5ex>[r]^-{X_-} & \cdots \ar@<.5ex>[l]^-{X_+}
\ar@<.5ex>[r]^-{X_-} & V_{-N+3} \ar@<.5ex>[l]^-{X_+}
\ar@<.5ex>[r]^-{X_-} & V_{-N+1} \ar@<.5ex>[l]^-{X_+} \ar[r]^-{X_-} & 0.
}
\end{equation}
\end{rmk}

Let us return to our case of characteristic $p=3$. For $D=3$,
i.e., 3\defis dimensional irreducible representations, we have more than
in the case of characteristic 0, where there is only the module
$\mthree$,
\begin{equation}
\xymatrix{
0 & V_{-1} \ar[l]^-{X_+} \ar@<.5ex>[r]^-{X_-} & V_0 \ar@<.5ex>[l]^-{X_+}
\ar@<.5ex>[r]^-{X_-} & V_1 \ar@<.5ex>[l]^-{X_+} \ar[r]^-{X_-} & 0.
}
\end{equation}
There is in fact an entire family of irreducible representations,
parametrized by a 3\defis dimensional variety, of which $\mthree$ is a
special case. Writing the images of the generators of $\sltwo$ as
matrices acting on a 3\defis dimensional vector space, these
representations are given as follows. First, we have the irreducible
modules that we denote by $T(b,c,d)$:
\begin{equation}\label{eq:T}
X_- = \mat{ 0 & 0 & c \\ 1 & 0 & 0 \\ 0 & 1 & 0 } \quad H = \mat{
d-1 & 0 & 0 \\ 0 & d & 0 \\ 0 & 0 & d+1 } \quad X_+ = \mat{ 0 &
a_1 & 0 \\ 0 & 0 & a_2 \\ b & 0 & 0 },
\end{equation}
where
\begin{equation}\label{a1a2}
\begin{aligned}
a_1 &= bc+d-1, \\
a_2 &= a_1+d = bc-d-1.
\end{aligned}
\end{equation}
We also have the family of ``opposite'' irreducible modules, where the
forms of $X_+$ and $X_-$ are exchanged, which we denote by
$\widetilde{T}(b,c,d)$:
\begin{equation}\label{eq:Ttilde}
X_- = \mat{ 0 & 0 & b \\ a_1 & 0 & 0 \\ 0 & a_2 & 0 } \quad H =
\mat{ d-1 & 0 & 0 \\ 0 & d & 0 \\ 0 & 0 & d+1 } \quad X_+ = \mat{
0 & 1 & 0 \\ 0 & 0 & 1 \\ c & 0 & 0 }.
\end{equation}
In both of these cases, $b$, $c$, and $d$ are arbitrary elements of
the ground field $k$, \emph{however} we don't allow the cases
\begin{equation}\label{eq:1}
T(0,0,1)\ \text{ or } \ T(0,0,-1),
\end{equation}
since in these cases the representation is not irreducible. Once
the other parameters are chosen, $a_1$ and $a_2$ are necessarily
given in terms of $b$, $c$, and $d$ by \eqref{a1a2} if the
matrices \eqref{eq:T} and \eqref{eq:Ttilde} are to be
representations of $\sltwo$. (To see this, one can explicitly
solve, for example, the equation $[X_+,X_-]=H$ for $a_1$ and $a_2$
using the above matrix representations, then check that the
relations $[H,X_{\pm}] = \pm 2 X_{\pm}$ are satisfied).

In addition to these two kinds of irreducible modules, we have each of
their duals, which we will denote by $T^* (b,c,d)$ and $\widetilde{T}^*
(b,c,d)$.

\begin{rmk}
Note that $\mthree \simeq T(0,0,0)$ as $\sltwo$-modules.
\end{rmk}

All of these irreducible modules can be glued into the following
indecomposables.

We let $M \c+ \widetilde{M}$ denote the \emph{semidirect sum} of the
subspaces $M$ and $\widetilde{M}$. By this, we mean that $\widetilde{M}$ is a
submodule.

A diagram of subspaces (e.g. $M \rightarrow \widetilde{M}$) indicates
something similar, but gives more information. A subspace that is the
source of no arrows (in our example, $\widetilde{M}$) is a submodule. A
subspace that is the source of some arrows becomes a submodule upon
taking the quotient modulo the targets of those arrows (in our
example, $M/\widetilde{M}$). Note that a subspace that is the source of
some arrows cannot be selected uniquely, since it is a quotient
space. Instead, we should think of it as the span of a collection of
vectors, the representatives of which form a basis for the quotient
space.

The direction of an arrow in a diagram also carries information. An
arrow pointing to the right indicates that we get an element of the
target by acting on an element of the source with $X_-$. An arrow
pointing to the left indicates the same for $X_+$. Note that the
``direction'' of the arrow refers only to the left/right
direction. That is, an arrow that points up/down and right is still
thought of as an ``arrow pointing to the right'', and similarly for up/down and
left. We also say, taking the example from the previous paragraph, that
``$M$ is glued in to $\widetilde{M}$ via $X_-$''.

For later, we let $M_1$ denote the submodule (cf.~\cite{leites-indecomposable}):
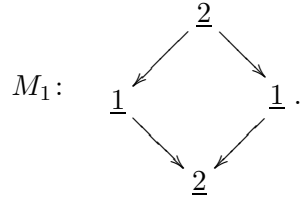
\begin{equation*}
\raisebox{-2.5em}{$M_1\colon$\quad}\xymatrix@dr{
 \mtwo \ar[d] \ar[r] & \mone\ . \ar[d] \\
 \mone \ar[r] & \mtwo
}
\end{equation*}
To make the statements of the last two paragraphs concrete, we dissect
this particular case. The symbol $\mtwo$ at the bottom indicates an
irreducible submodule. The symbol $\mone$ on the left is represented
by the span of a single vector $v$ with $X_+ v = 0$ and $X_- v$ a
vector of the irreducible submodule $\mtwo$ at the bottom. Similarly, the
symbol $\mone$ on the left is represented by the span of a single
vector $w$ with $X_- w = 0$ and $X_+ w$ a vector of the irreducible
submodule $\mtwo$ at  the bottom. Finally, the symbol $\mtwo$ at the top stands
for the span of two vectors, $v'$ and $w'$, with
\begin{equation}
\begin{aligned}
X_- v' &= w',  &  X_- w' &= \mu w, \\
X_+ v' &= \lambda v,  &  X_+ w' &= v',
\end{aligned}
\end{equation}
for some $\lambda, \mu \in k$, i.e. $X_+ v'$ is contained in the left
``$\mone$'' and $X_- w'$ is contained in the right ``$\mone$''.

The main result of the paper is the following theorem.

\begin{thm}\label{thm:main}
The decompositions (into indecomposable submodules) of tensor products
of all irreducible $\sltwo$-modules ($\mone$, $\mtwo$, $T(b,c,d)$,
$\widetilde{T}(b,c,d)$, and their duals) are completely described by
\begin{enumerate}
\item $\mone \otimes V \simeq V$ for any $\sltwo$-module $V$,
\item $\mtwo \otimes \mtwo = \mone \oplus \mthree$,
\item $\mtwo \otimes \widetilde{T}(b,\frac{1}{b},0) =
  \widetilde{T}(b,\frac{1}{b},0) \oplus T(\frac{1}{b},0,0)$,
\end{enumerate}
together with Tables \ref{table:TwoTimesbcd}--\ref{table:bcdTimesbcd},
found at the end of the paper.
\end{thm}

\begin{rmk}
Statement (1) of the theorem is obvious.
\end{rmk}

\begin{rmk}
In \S \ref{sec:prelim}, we will explain how we arrived at the
list of modules examined in Tables
\ref{table:TwoTimesbcd}--\ref{table:bcdTimesbcd}, as well as why we are
allowed to seemingly ignore certain modules.
\end{rmk}

The paper is organized as follows. In \S \ref{sec:prelim} we will
closely examine the families of modules and reduce the number of
different cases we must consider separately by demonstrating certain
correspondences. Then, we will study the structure of the modules we
are concerned with and, in particular, semidirect sums. In \S
\ref{sec:not}, we will briefly present the notation we use for the
calculations. Finally, in \S \ref{sec:cases}, we break down the
various tensor products we can form case by case and compute their
decompositions.

All computations for this paper were made with the assistance of
\texttt{SuperLie} (\cite{superlie}, \cite{grozman-leites-superlie}).

\section{Preliminaries}\label{sec:prelim}

Before we get into the meat of the paper, let us carefully describe
the $\sltwo$-modules in characteristic 3, first the irreducible ones,
then certain indecomposables (to describe all indecomposables is an
open problem).

We begin by proving a couple of lemmas which will help us to reduce
the amount of work we have to do. In particular, we can show (via a
change of basis) that some seemingly different modules are actually
isomorphic. In the end, we will only have to consider the cases given in
the following lemma (in addition, of course, to $\mtwo$):

\begin{lem}\label{lemma:families}
We can represent every 3\defis dimensional irreducible
$\sltwo$-module by a member of one of the following two families:
\begin{enumerate}
\item $T(b,c,d)$, where $b,c,d \in k$ are arbitrary subject to \eqref{eq:1};
\item $\widetilde{T}(b,\frac{1}{b},0)$.
\end{enumerate}
\end{lem}

The proof of Lemma \ref{lemma:families} is immediately implied by
Lemmas \ref{lemma:T}, \ref{lemma:Ttilde}, and \ref{lemma:twozeros} below.

\begin{lem}\label{lemma:T}
For the dual modules of $T(b,c,d)$ and $\widetilde{T}(b,c,d)$, given any
$b,c,d \in k$, we have:
\begin{enumerate}
\item $T^* (b,c,d) \simeq T(b',c',d')$ for some $b',c',d' \in k$,
\item $\widetilde{T}^* (b,c,d) \simeq \widetilde{T}(b',c',d')$ for some
  $b',c',d' \in k$.
\end{enumerate}
\end{lem}

\begin{proof}
1) Recall that, given a matrix representation $X$ of the action of
an element of a Lie algebra on a module, the action of $X$ on the
dual module is given by $-X^t$, i.e., the negative transpose of
the original matrix. So here, the action of $\sltwo$ on the dual
module $T^* (b,c,d)$ is as follows:
\begin{equation}
 X_- = \mat{
 0 & -1 & 0 \\
 0 & 0 & -1 \\
 -c & 0 & 0
 },
 \quad
 H = \mat{ -d+1 & 0 & 0 \\ 0 & -d & 0 \\ 0 & 0 & -d-1 },\quad
 X_+ = \mat{ 0 & 0 & -b \\ -a_1 & 0 & 0 \\ 0 & -a_2 & 0 }.
\end{equation}
We apply a similarity transformation given by the matrix
\begin{equation}\label{eq:simtrf}
S = \mat{ 0 & 0 & 1 \\ 0 & -1 & 0 \\ 1 & 0 & 0 }.
\end{equation}
We then rename the resulting matrices to $X_{\pm}' := S X_\pm S^{-1}$
and $H' := S H S^{-1}$:
\begin{equation}\label{eq:trfmat}
X'_- = \mat{ 0 & 0 & c' \\ 1 & 0 & 0 \\ 0 & 1 & 0 }, \quad H' =
\mat{ d'-1 & 0 & 0 \\ 0 & d' & 0 \\ 0 & 0 & d'+1 }, \quad X'_+ =
\mat{ 0 & a'_1 & 0 \\ 0 & 0 & a'_2 \\ b' & 0 & 0 }.
\end{equation}

Here one easily checks that $a'_1 = a_2$, $a'_2 = a_1$, $b' = -b$, $c'
= -c$, and $d' = -d$. We have transformed our original
representation $T^* (b,c,d)$ into one of the form $T(b',c',d')$.

2) The action on the dual module $\widetilde{T}^* (b,c,d)$ is
given by:
\begin{equation}
 X_- = \mat{ 0 & -a_1 & 0 \\ 0 & 0 & -a_2 \\ -b & 0 & 0 }, \quad
 H = \mat{ -d+1 & 0 & 0 \\ 0 & -d & 0 \\ 0 & 0 & -d-1 }, \quad
 X_+ = \mat{ 0 & 0 & -c \\ -1 & 0 & 0 \\ 0 & -1 & 0 }.
\end{equation}
Since this is in complete analogy with the previous case, we skip the
detailed calculation. (We do note that one can, in fact, even use the same
similarity tranformation as in \eqref{eq:simtrf} above.)
\end{proof}

\begin{lem}\label{lemma:Ttilde}
For an appropriate choice of $b',c',d' \in k$ we have
\begin{equation}
\widetilde{T}(b,c,d) \simeq T(b',c',d'),
\end{equation}
if and only if at
most one of $a_1, a_2,$ and $b$ are equal 0.
\end{lem}
\begin{proof}
First, we note that if two of $a_1, a_2,$ and $b$ are zero, then we
will clearly get no isomorphism, since for modules of the form
$T(b,c,d)$, the matrix of $X_-$ can have at most one eigenvector with
eigenvalue 0.

The proof of the converse statement is broken up into different cases.

1) $a_1, a_2 \neq 0$. We apply the similarity transformation with
the matrix
\begin{equation}
S = \mat{ \frac{1}{a_1} & 0 & 0 \\ 0 & 1 & 0 \\ 0 & 0 &
\frac{1}{a_1 a_2} }.
\end{equation}
This yields matrices $H'$, $X'_\pm$ of the same form as in
\eqref{eq:trfmat} above. In this case, using the same notation as
above, we have $a'_1 = a_1$, $a'_2 = a_2$, $b' = \frac{c}{a_1 a_2}$,
$c' = a_1 a_2 b$, and $d' = d$.

2) $a_1 = 0$, $a_2, b \neq 0$. We apply the similarity transformation with
the matrix
\begin{equation}
S = \mat{ 0 & 1 & 0 \\ 0 & 0 & \frac{1}{a_2} \\ \frac{1}{b a_2} &
0 & 0 }.
\end{equation}
We get matrices as in the form of \eqref{eq:trfmat}, where this time
$a'_1 = a_2$, $a'_2 = bc$, $b' = \frac{1}{b a_2}$, $c' = 0$, and $d' =
d+1$.

3) $a_2 = 0$, $a_1,b \neq 0$. We apply the similarity
transformation with the matrix
\begin{equation}
S = \mat{ 0 & 0 & 1 \\ \frac{1}{b} & 0 & 0 \\ 0 & \frac{1}{b a_1}
& 0 }.
\end{equation}
We again get matrices in the form of \eqref{eq:trfmat}, where
$a'_1 = bc$, $a'_2 = a_1$, $b' = \frac{1}{b
a_1}$, $c' = 0$, and $d' = d-1$.
\end{proof}

The following lemma will tell us more about $\widetilde{T}(b,c,d)$ when
two of $a_1$, $a_2$, and $b$ are zero.

\begin{lem}\label{lemma:twozeros}
If two of $a_1$, $a_2$, and $b$ are zero, then $\widetilde{T}(b,c,d)\simeq\widetilde{T}(b,\frac{1}{b},0)$.
\end{lem}
\begin{proof}
There are three possibilities: $a_1, a_2 = 0$, $a_1, b = 0$, and $a_2,
b = 0$. Let us examine each one in turn.

1) $a_1, a_2 = 0$. In this case, we have
\begin{equation}
\begin{split}
bc+d-1 & = 0, \\
bc-d-1 & = 0.
\end{split}
\end{equation}
Subtracting the second equation from the first implies $d=0$, which we
can plug back into the first equation to get $bc-1=0$, or
$c=\frac{1}{b}$. Hence, the module must be of the form $\widetilde{T}(b,\frac{1}{b},0)$.

2) $a_1, b = 0$. Here we have $0 = a_1 = bc+d-1 = d-1$. Hence,
$d=1$. Furthermore, by Lemma \ref{lemma:dint} below, we can
transform $T(0,c,1)$ into a representation with $d=0$ provided
that $c \neq 0$. Since $X_-$ will still have two eigenvectors with
eigenvalue 0, the transformed representation will necessarily be
as in case 1). (When $a_2=0$ and $b =0$, $d$ is necessarily $-1$; see
case 3) below.) Therefore, the only truly new case is
$\widetilde{T}(0,0,1)$. However, $\widetilde{T}(0,0,1)$ is not
irreducible (see the Introduction) and, since we assumed our
module to be irreducible, is disallowed. So we have now completely
reduced to case 1).

3) $a_2, b = 0$. In this case, we get $0 = a_2 = bc-d-1 = -d-1$,
or $d=-1$. As above, we can reduce this to case 1) if and only if
$c \neq 0$. Therefore, the only new case is
$\widetilde{T}(0,0,-1)$. However, from the Introduction we
know that $\widetilde{T}(0,0,-1)$ is not irreducible, and as in
case 2) we have now completely reduced to case 1).
\end{proof}

We also want to examine the cases where $d =0$ or $\pm 1$ in more
detail, since they will turn out to be special once we start
tensoring. It turns out that all three of these cases correspond to a
module where $d=0$ unless $c=0$:

\begin{lem}\label{lemma:dint}
Let $d=\pm 1$ and $c \neq 0$. Then
\begin{equation}
T(b,c,d) \simeq T(b',c',0)
\end{equation}
for an appropriate choice of $b'$ and $c'$.

The statement remains true if we replace $T$ by
$\widetilde{T}$ everywhere above.
\end{lem}
\begin{proof}
We will prove the statement for $T(b,c,1)$ (i.e. for $d=1$). The other cases
are completely analogous.

In this case, our representation is given by the following matrices:
\begin{equation}
 X_- = \mat{ 0 & 0 & c \\ 1 & 0 & 0 \\ 0 & 1 & 0 }, \quad
 H = \mat{ 0 & 0 & 0 \\ 0 & 1 & 0 \\ 0 & 0 & -1 }, \quad
 X_+ = \mat{ 0 & a_1 & 0 \\ 0 & 0 & a_2 \\ b & 0 & 0 }.
\end{equation}

We apply the similarity transformation with matrix
\begin{equation}
S = \mat{ 0 & 0 & 1 \\ \frac{1}{c} & 0 & 0 \\ 0 & \frac{1}{c} & 0
},
\end{equation}
renaming the resulting matrices to $X_{\pm}' := S X_\pm S^{-1}$
and $H' := S H S^{-1}$ to get
\begin{equation}
X'_- = \mat{ 0 & 0 & c' \\ 1 & 0 & 0 \\ 0 & 1 & 0 } \quad H' =
\mat{ -1 & 0 & 0 \\ 0 & 0 & 0 \\ 0 & 0 & 1 } \quad X'_+ = \mat{ 0
& a'_1 & 0 \\ 0 & 0 & a'_2 \\ b' & 0 & 0 }.
\end{equation}
Explicitly, we have $a'_1 = bc$, $a'_2 = a_1$, $b' = \frac{a_2}{c}$,
and $c' = c$. So the change of basis gives us a representation of the
form $T(b',c',0)$.
\end{proof}

Now that we have proved Lemma \ref{lemma:families}, we move on to
studying semidirect sums and the structure of our
$\sltwo$-modules. Let $S$ and $T$ be two 3\defis dimensional
irreducible $\sltwo$-modules, and let us consider the tensor product
$V=S \otimes T$ as an $\sltwo$-module. This is a 9\defis dimensional
space, but we may divide it into three distinguished 3\defis dimensional
subspaces, the \emph{weight spaces}, i.e., the eigenspaces of $H$. (It
is simple to check that $H$ has three distinct eigenvalues for any
such $S$ and $T$.)

\begin{rmk} 
For the remainder of this section, all vectors are
assumed to be weight vectors, i.e., eigenvectors of $H$.
\end{rmk}

Thus, instead of considering the full 9\defis dimensional space $V$,
we restrict attention to one of the 3\defis dimensional
weight spaces. We denote them by
\begin{equation}
V_\rho = \left\{ v \in V \mid Hv = \rho v \right\}.
\end{equation}

The element $X_+ X_- \in U(\sltwo)$ acts on each weight space, since if
$Hv=\rho v$ for some $\rho \in k$, then
\begin{equation}
\begin{aligned}
H(X_+ X_- v) &= X_+ H X_- v + [H,X_+] X_- v \\
 &= X_+ X_- H v + X_+ [H,X_-] v + 2 X_+ X_- v \\
 &= \rho X_+ X_- v.
\end{aligned}
\end{equation}
Therefore, we may consider $X_+ X_-$ as a linear transformation of one
such space and look for its eigenvalues and eigenvectors. We relate
this to semidirect sums in the following lemma:

\begin{lem}\label{lemma:eigenvalues}
Let $X_+ X_-$ have two distinct eigenvalues, $\lambda_1$ and
$\lambda_2$, on the subspace $V_\rho$ for some eigenvalue $\rho$ of
$H$. Suppose that $V$ contains a semidirect sum $M \c+
\widetilde{M}$. Further, let $\widetilde{M} \cap V_\rho = \mathrm{span} (v_1,v_2)$,
where $X_+ X_- v_i = \lambda_i v_i$. Consider the action of $\sltwo$ on
the quotient space and a vector $m \in M/\widetilde{M}$ with $Hm = \rho
m$. Then $X_+ X_- m = \lambda_i m$ for some $i$.
\end{lem}
\begin{proof}
Suppose, on the contrary, that $X_+ X_- m = \mu m$ (equality being in the
quotient space) for some $\mu \neq \lambda_i$ for all $i$. We will
show that there is a vector $v \in V$ with $Hv = \rho v$ and $X_+ X_-
v = \mu v$, a contradiction.

For the remainder of the proof, the action of $\sltwo$ will be on the
full space $V$, not the quotient space.

We know that $X_+ X_- m = \mu m + \widetilde{m}$ for some $\widetilde{m} \in
\widetilde{M}$. By the assumptions of the lemma, we can write $\widetilde{m} =
\widetilde{m}_1 + \widetilde{m}_2$, where $X_+ X_- \widetilde{m}_i = \lambda_i
\widetilde{m}_i$. We set
\begin{equation}
v = m + \frac{1}{\mu-\lambda_1} \widetilde{m}_1 + \frac{1}{\mu-\lambda_2}
\widetilde{m}_2.
\end{equation}
Since $\mu \neq \lambda_1$ or $\lambda_2$, $v$ is well-defined, and it is
easily seen that $X_+ X_- v = \mu v$.
\end{proof}

We can consider other eigenvalue equations on the weight spaces of a
tensor product, in particular for $X_+^3$ and $X_-^3$. The proof of
the following lemma is straightforward and is left to the reader.

\begin{lem}\label{lemma:cubed}
The action of $X_+^3$ and $X_-^3$ on any weight vector $v$ in the
module $V = S \otimes T$, where $S$ and $T$ are of the form
$\widetilde{T}(b,\frac{1}{b},0)$ or $T(b,c,d)$, is given by:
\begin{enumerate}
\item $\widetilde{T}(b,\frac{1}{b},0) \otimes
  \widetilde{T}(\beta,\frac{1}{\beta},0)$: $X_+^3v =
  \frac{b+\beta}{b\beta}$, $X_-^3 v = 0$
\item $\widetilde{T}(b,\frac{1}{b},0) \otimes T(\beta,\gamma,\delta)$:
  $X_+^3v = (\frac{1}{b}+\beta \alpha_1 \alpha_2) v$, $X_-^3 v =
  \gamma v$
\item $T(b,c,d) \otimes T(\beta,\gamma,\delta)$: $X_+^3v =
  (b a_1 a_1 + \beta \alpha_1 \alpha_2) v$, $X_-^3 v = (c+\gamma) v$
\end{enumerate}
Furthermore, assuming $v$ is some weight vector, $V$ contains highest
weight vectors if and only if $X_+^3 v = 0$, and lowest weight vectors
if and only if $X_-^3 v = 0$.
\end{lem}

Combining Lemmas \ref{lemma:eigenvalues} and \ref{lemma:cubed}, we get:

\begin{lem}\label{lemma:params}
Given a module $T(b,c,d)$ with no lowest weight vectors (i.e.~$c \neq
0$), or a module $\widetilde{T}(b,\frac{1}{b},0)$, we can determine $b$, as
well as $c$ and $d$ (where applicable), from the actions of $X_+ X_-$
and $X_-^3$ on the weight spaces. If $c=0$ in $T(b,c,d)$, we may still
determine $b$ from the action of $X_+^3$ if we know $d$.
\end{lem}
\begin{proof}
Let us fix a basis $\{v_1 , v_2 , v_3 \}$ such that the
$\sltwo$-action is given by the matrices of \eqref{eq:T} or
\eqref{eq:Ttilde}.

For $\widetilde{T}(b,\frac{1}{b},0)$ we must simply note that $X_+^3 v_1 =
\frac{1}{b} v_1$.

For $T(b,c,d)$, where $c \neq 0$ we note that the equations
\begin{equation}
 \begin{split}
 &X_+ X_- v_1 = (bc-1+d) v_1, \\
 &X_+ X_- v_3 = bc v_3, \\
 &X_-^3 v_1 = c v_1
\end{split}
\end{equation}
allow us to determine $b$, $c$, and $d$, since $c \neq 0$.

In the case of $T(b,0,d)$ where $d$ is known, we note that
\begin{equation}
X_+^3 v_1 = b a_1 a_2 v_1 = b (1-d^2) v_1,
\end{equation}
allowing us to determine $b$.
\end{proof}

The number of linearly independent eigenvectors of $X_+ X_-$ will be
important to us when determining the structure of a decomposition, as
the following lemma shows.

\begin{lem}\label{lemma:decomp}
Let there be either no highest or no lowest weight vectors in
$V$.

{\em 1)} If $X_+ X_-$ has three linearly independent eigenvectors $v_1$, $v_2$,
and $v_3$ in $V_\rho$ for an eigenvalue $\rho$ of $H$, then there
are 3\defis dimensional irreducible submodules $M,M',M'' \subset V$ such
that $V = M \oplus M' \oplus M''$.

{\em 2)} If $X_+ X_-$ has only two linearly independent eigenvectors $v_1$ and
$v_2$ in $V_\rho$, then there are
3\defis dimensional submodules $M,M',M'' \subset V$ such that $V = M \c+ (M'
\oplus M'')$. In this case, $M'$ and $M''$ are irreducible.
\end{lem}

\begin{rmk}
In the cases we consider, $X_+ X_-$ will always have at least two
distinct eigenvectors on any $V_\rho$.
\end{rmk}

\begin{proof}
We will prove this for the case where $X_+ v \neq 0$ for all $v \in
V$. The other case is analogous.

Clearly, in an irreducible submodule of $\sltwo$, for any weight
vector $v$, we have $X_+ X_- v = \lambda v$ for some $\lambda \in
k$. Further, a one-dimensional subspace belongs to at most one
irreducible submodule. So we can have at most as many irreducible
submodules as we have eigenvectors of $X_+ X_-$ in some
$V_\rho$.

Furthermore, if $X_+ X_- v = \lambda v$ for some $\lambda \in k$ and
$v$ a weight vector, then $v$, $X_+ v$, and $X_+^2 v$ form an
irreducible submodule. From Lemma \ref{lemma:cubed}, it is clear that
$X_+^3 v$ is a nonzero multiple of $v$. From the relations
of $\sltwo$, it is easy to check that $X_- X_+^n v$ is a multiple of
$X_+^{n-1} v$ for any nonnegative integer $n$. So these three vectors do
form a submodule, and it is irreducible because $X_+^n$ sends any
subspace to any other for some $n$.

This completes the proof of heading 1), since in that case we
can build three 3\defis dimensional irreducible submodules by the
procedure in the last paragraph. For heading 2), we can only
build two such irreducible submodules $M'$ and $M''$ this way. To
complete the proof, we consider $M = V / (M' \oplus M'')$ and the
action of $\sltwo$ on this quotient. We take any nonzero weight vector
$m \in M$ and note again that $m$, $X_+ m$, and $X_+^2 m$ form a basis
of this quotient module, and that it is irreducible. They are nonzero
because $X_+^3 m$ is a nonzero multiple of $m$. They are linearly
independent because they are of different weights, and they form an
irreducible submodule because there is no invariant subspace ($X_+^n$
sends any subspace to any other for some $n$).
\end{proof}

\section{Notation}\label{sec:not}

For the sake of convenience, we fix the notation for our modules for
the rest of the paper. We denote vectors in our modules according to
the following scheme:
\begin{equation}
\begin{split}
 &q_i \in \mtwo\ \text{ for } i=1,2; \\
 &t_i \in \widetilde{T} (b,\frac{1}{b},0),\ u_i \in \widetilde{T}
 (\beta,\frac{1}{\beta},0)\ \text{ for } i=1,2,3; \\
 &v_i \in T(b,c,d),\ w_i \in T(\beta,\gamma,\delta)\ \text{ for }
 i=1,2,3.
\end{split}
\end{equation}
The index in the subscript refers to the weight of the vector. For
$\mtwo$, the vector $q_1$ is of weight 1 and $q_2$ is of weight $-1$. For $t_i$,
$u_i$, $v_i$, and $w_i$, we refer to the matrix representations
\eqref{eq:T} and \eqref{eq:Ttilde} and set $t_i=e_i$,
$u_i=e_i$, etc., where $\{e_i \mid i=1,2,3\}$ is the usual basis of
$k^3$.

In addition, since we will need them so often and the notation becomes
awkward, we will omit the ``$\otimes$'' when writing tensor products
of vectors. For example, $q_1 \otimes v_2 \in \mtwo \otimes T(b,c,d)$
will be expressed simply as $q_1 v_2$.

\section{Case-by-case calculations}\label{sec:cases}

All calculations were done with the assistance of the
\texttt{Mathematica}-based package \texttt{SuperLie} (\cite{superlie},
\cite{grozman-leites-superlie}).

Due to space considerations, and for the flow of arguments,
calculations will not be repeated here in detail. Instead, we refer to
the \texttt{Mathematica} notebooks, which have been made available
online at:
\begin{equation*}
\href{http://personal-homepages.mis.mpg.de/clarke/Tensor-Calculations.tar.gz}{\texttt{http://personal-homepages.mis.mpg.de/clarke/Tensor-Calculations.tar.gz}}
\end{equation*}

We have focused on producing explicit decompositions where
possible. Many decompositions could have also been deduced using more
general arguments based on the lemmas of \S\ref{sec:prelim}, as was
done in \ref{sec:bc0bcdNothing} or \ref{sec:bcdbcdNothing}.

\subsection{The case $V = \mtwo \otimes \mtwo$}\label{subsec:2times2}

This is the simplest case of all; in fact, there is no difference from
the decomposition in characteristic 0. We start from the two highest
weight vectors, $r_1 r_1$ and $r_2 r_1 - r_1 r_2$. Applying $X_-$
gives an irreducible submodule of the form $\mthree$ from the first,
and $\mone$ from the second. Hence, we have a direct sum of two
irreducible modules,
\begin{equation}
\mtwo \otimes \mtwo = \mthree \oplus \mone.
\end{equation}

This could have been deduced from more general considerations as
well. Working over fields of characteristic $p$, if we have two
modules of the form $\underline{N}$ and $\underline{M}$, where
$N,M \in \{1,\ldots,p\}$ (defined in \eqref{eq:N}), they have
highest weights $N-1$ and $M-1$, respectively. Taking their tensor
product $\underline{N} \otimes \underline{M}$ will give us vectors
with weights at most $N+M-2$. If $N+M-2 < p$, then we will see
behavior that is no different from the case of characteristic~0.
Things start getting interesting when $N+M-2 \geq p$, which we
will have in all following cases.

\subsection{The case $V = \mtwo \otimes \widetilde{T}(b,\frac{1}{b},0)$}

In this case, we find two highest weight vectors, $q_2 t_1$, and $q_2
t_2$. Applying $X_-$ to these, we construct two 3\defis dimensional
irreducible submodules that supply a complete decomposition. From $q_2
t_1$, we get the module $T(\frac{1}{b},0,0)$, and from $q_2 t_2$, we
get the module $\widetilde{T}(b,\frac{1}{b},0)$. So
\begin{equation}
\mtwo \otimes \widetilde{T}\Big(b,\frac{1}{b},0\Big) =
T\Big(\frac{1}{b},0,0\Big) \oplus
\widetilde{T}\Big(b,\frac{1}{b},0\Big).
\end{equation}

\subsection{The case $V = \mtwo \otimes T(b,c,d)$}

In this case, we have the lowest weight vectors $q_2 v_3$ and $q_1 v_3
- q_2 v_2$ if and only if $c=0$. We have highest weight vectors if and
only if we are in one of the following situations:
\begin{enumerate}
\item $q_1 v_1$ and $q_1 v_2 + (1-d) q_2 v_1 \Longleftrightarrow b=0$,
\item $q_1 v_3$ and $q_1 v_1 - b q_2 v_3 \Longleftrightarrow 1-bc+d=0$,
\item $q_1 v_2$ and $q_1 v_3 + (1-bc+d) q_2 v_2 \Longleftrightarrow 1-bc-d=0$.
\end{enumerate}
The ``only if'' part of these statements is provided by Lemma
\ref{lemma:cubed}.

Furthermore, the eigenvalues of $H$ are $d$ and $d \pm 1$. Hence,
whenever $d=0$, 1, or 2, we will have different behavior---e.g., the
appearance of submodules like $\mthree$ and $\mone \rightarrow
\mtwo$.

Therefore, we divide our computations into subcases. For the precise
breakdown of these subcases, we note that for $c=0$, whether we have
the highest weight vectors described in (2) and (3) above depends only
on $d$---we have (2) if and only if $d=-1$ and (3) if and only if
$d=1$. Furthermore, if $c \neq 0$, then we need not consider $d=\pm 1$
separately from $d=0$, since in both of these cases $T(b,c,d)$ is
isomorphic to a module with $T(b',c',0)$ for some $b'$ and $c'$ by
Lemma \ref{lemma:dint}.

\subsubsection{The subcase $c=0$; $d=0$; $b=0$}
In this case, by acting on the two lowest and highest weight vectors
listed above, we compute a submodule of the form $\mone \to \mtwo
\gets \mone$. We quickly verify that this contains all highest and
lowest weight vectors.

The full module has dimension six, so we do not yet have a complete
decomposition. Since any weight vectors not contained in $\mone \to
\mtwo \gets \mone$ have weights 1 and $-1$, we know that they will
form a module of the form $\mtwo$ after quotienting by $\mone \to
\mtwo \gets \mone$ (to form two modules $\mone$, they would have to
have weight 0). A direct computation shows that all together, the
module decomposes as $M_1$.

\subsubsection{The subcase $c=0$; $d=0$; $b \neq 0$}\label{sec:2bcd00}
Here, we only have lowest weight vectors to work with. Acting on them
by $X_+$ gives rise to one irreducible submodule,
$M:=\widetilde{T}(\frac{1}{b},b,0)$, which contains both lowest weight
vectors.

Since there are no more highest or lowest weight vectors, there are
now two possibilities. Either the remainder of the full module forms a
submodule of the form $T(b',c',d')$ with no highest or lowest weight
vectors, or the remainder forms irreducible submodules only upon
taking some quotients.

The first possibility can hold if and only if $X_+ X_-$ has a basis of
eigenvectors for each weight space. A quick check shows that the
minimal polynomial of $X_+ X_-$ is $(\lambda-1)^2$ when acting on the
space of weight 1 vectors. Hence, the module contains no more
irreducible submodules.

A direct computation shows us now that the quotient module $V/M$ is
again of the form $\widetilde{T}(\frac{1}{b},b,0)$, so we get the complete
decomposition
\begin{equation}
\widetilde{T}\Big(\frac{1}{b},b,0\Big) \c+
\widetilde{T}\Big(\frac{1}{b},b,0\Big).
\end{equation}

\subsubsection{The subcase $c=0$; $d=1$; $b=0$}\label{sec:2bcd010}
We begin by acting on the lowest weight vectors by $X_+$, which gives
us two irreducible submodules, $\mthree$ and $\mone$. This exhausts
the lowest weight vectors, yet $V$ is not yet completely
decomposed \Tire we are missing a 2\defis dimensional subspace.

Since $1-bc-d=1-d=0$, there is one remaining highest weight
vector. Acting on it by $X_-$ gives us the final submodule $\mtwo$, which is
glued into $\mone$ via $X_-$. Hence, we get
\begin{equation}
\mthree \oplus
(\mtwo \to \mone).
\end{equation}

\subsubsection{The subcase $c=0$; $d=1$; $b \neq 0$}\label{sec:2bcd01}
As in the last case, we immediately get two irreducible submodules,
$\mthree$ and $\mone$. Here, there are no remaining highest weight
vectors, but we note that $\xpm$ has a basis of eigenvectors for the
space of weight 1 vectors. We use this to compute that the remaining
2\defis dimensional submodule is irreducible after quotienting by
$\mone$, so we get
\begin{equation}
\mthree \oplus (\mtwo \c+ \mone).
\end{equation}

\subsubsection{The subcase $c=0$; $d=2$; $b=0$}
This case proceeds exactly analogously to Section \ref{sec:2bcd010}. However,
since we have different highest weight vectors (here $1-bc+d=1+d=0$),
we get the slightly different decomposition
\begin{equation}
\mthree
\oplus (\mone \to \mtwo).
\end{equation}

\subsubsection{The subcase $c=0$; $d=2$; $b \neq 0$}
As above, this is completely analogous to Section \ref{sec:2bcd01},
but we get the decomposition
\begin{equation}
\mthree \oplus (\mone \c+ \mtwo).
\end{equation}

\subsubsection{The subcase $c=0$; $d \neq 0,1,2$}
In this case, we may have highest weight vectors if $b=0$, but it
turns out that in either case, acting on the lowest weight vectors by $X_+$
immediately gives us a complete decomposition,
\begin{equation}
T\Big(\frac{b(d-1)}{d},0,d-1\Big) \oplus
T\Big(\frac{b(d+1)}{d},0,d+1\Big).
\end{equation}

\subsubsection{The subcase $c \neq 0$; $d=0$; $b=0$}
By acting on the highest weight vectors with $X_-$, we get a
module of the form $M := T(0,c,1)$, which exhausts the highest weight
vectors. Since there are no more highest or lowest weight vectors, we
have the same situation as in Section \ref{sec:2bcd00}. As there, we
can check that $\xpm$ does not have a basis of eigenvectors for the
space of weight 1 vectors. We then directly compute that $V/M$ is again
of the form $T(0,c,1)$. In all, we get
\begin{equation}
T(0,c,1) \c+ T(0,c,1)
\end{equation}

\subsubsection{The subcase $c \neq 0$; $d=0$; $b=\frac{1}{c}$}
In this case there are three highest weight vectors, since
$1-bc+d=1-bc-d=0$. By acting on these with $X_-$, we get a complete
decomposition,
\begin{equation}
T(0,c,1) \oplus T(0,c,1).
\end{equation}

\subsubsection{The subcase $c \neq 0$; $d=0$; $b \neq 0$ or $\frac{1}{c}$}\label{sec:2bcd0}
Here there are no highest or lowest weight vectors. However, the
minimal polynomial for $\xpm$ acting on the space of weight 0 vectors
is $\lambda^2 + bc\lambda + bc(bc-1)$, which has the two distinct
roots $bc \pm \sqrt{bc}$. Hence, $\xpm$ has a basis of
eigenvectors for this space. Explicitly solving for these eigenvectors
and acting on them by $X_{\pm}$ gives the decomposition
\begin{equation}
 T\Big(b+\sqrt{\frac{b}{c}},c,1\Big) \oplus
 T\Big(b-\sqrt{\frac{b}{c}},c,1\Big).
\end{equation}

\subsubsection{The subcase $c \neq 0$; $d \neq 0,1$ or $2$; $b=0$}
In this case, acting by $X_-$ on the two highest weight vectors gives
us a complete decomposition,
\begin{equation}
T(0,c,d+1) \oplus T(0,c,d-1).
\end{equation}

\subsubsection{The subcase $c \neq 0$; $d \neq 0,1$ or $2$; $1-bc+d=0$}
Here again, acting by $X_-$ on the two highest weight vectors gives
us a complete decomposition,
\begin{equation}
T(0,c,d) \oplus T(0,c,d+1).
\end{equation}

\subsubsection{The subcase $c \neq 0$; $d \neq 0,1$ or $2$; $1-bc-d=0$}
Here again, acting by $X_-$ on the two highest weight vectors gives
us a complete decomposition,
\begin{equation}
T(0,c,d-1) \oplus T(0,c,d).
\end{equation}

\subsubsection{The subcase $c \neq 0$; $d \neq 0,1$ or $2$; $b \neq
  0$; $1-bc\pm d \neq 0$}
As in Section \ref{sec:2bcd0}, there are no highest
or lowest weight vectors. Investigating the action of $\xpm$ on the
space of weight $d$ vectors shows that it has eigenvalues $bc-d \pm
\sqrt{bc+d^2}$, and that $\xpm$ has a basis of eigenvectors for this
space if and only if the eigenvalues are distinct, i.e., if $bc+d^2 \neq 0$.

If $bc+d^2 \neq 0$, we solve for these eigenvectors and act on them by
$X_{\pm}$ to get a complete decomposition,
\begin{equation}
T\Big(b+\frac{d+\sqrt{bc+d^2}}{c},c,d+1\Big) \oplus
T\Big(b+\frac{d-\sqrt{bc+d^2}}{c},c,d+1\Big).
\end{equation}

If $bc+d^2=0$, the two eigenvectors from above degenerate to one, and
acting on it we get a submodule $M := T(b+\frac{d}{c},c,d+1)$. We can then
explicitly compute that $V/M = T(b+\frac{d}{c},c,d+1)$ as well. So in this
case, the complete decomposition is
\begin{equation}
T\Big(b+\frac{d}{c},c,d+1\Big) \c+
T\Big(b+\frac{d}{c},c,d+1\Big).
\end{equation}

\subsection{The case $V = \widetilde{T}(b,\frac{1}{b},0) \otimes
  \widetilde{T}(\beta,\frac{1}{\beta},0)$}

In this case, by Lemma \ref{lemma:cubed}, there are always lowest
weight vectors, namely $t_1 u_1$, $t_1 u_2$, $t_2 u_1$, $t_1 u_3 + t_3
u_1$, and $t_2 u_2$. We have the highest weight vectors $b t_1 u_1 +
t_2 u_3 - t_3 u_2$, $b t_1 u_2 - b t_2 u_1 - t_3 u_3$, and $t_1 u_3 -
t_2 u_2 + t_3 u_1$ if and only if $\frac{b+\beta}{b\beta}=0$, or
equivalently $\beta=-b$. Therefore, we need to split this case into
two subcases.

\subsubsection{The subcase $\beta=-b$}
Here, we start by acting on the highest weight vectors by $X_-$. This
immediately generates three irreducible modules, $\mthree$, $\mone$,
and $\mtwo$. This exhausts all of the highest weight vectors.

Since the full module has two lowest weight vectors not contained in
the above submodules, there are two possibilities. Either these lowest
weight vectors make up a module of the form
$\widetilde{T}(b',\frac{1}{b'},0)$, or they form irreducible modules
only upon quotienting. The first possibility is ruled out since
$X_+^3$ of any weight vector is zero, which is not the case for any
$\widetilde{T}(b',\frac{1}{b'},0)$. Hence, we have the second
possibility.

By acting on the remaining lowest weight vectors by $X_+$, we obtain
a module $\mtwo$ glued into the module $\mone$ via $X_+$, and a
module $\mone$ glued into the module $\mtwo$ via $X_+$. Hence the
full decomposition is
\begin{equation}
\mthree \oplus (\mtwo \gets \mone) \oplus
(\mone \gets \mtwo).
\end{equation}

\subsubsection{The subcase $\beta \neq -b$}
In this case, acting on the lowest weight vectors by $X_+$ immediately
gives us a complete decomposition,
\begin{equation}
T\Big(\frac{b+\beta}{b\beta},0,0\Big) \oplus
\widetilde{T}\Big(\frac{b\beta}{b+\beta},\frac{b+\beta}{b\beta},0\Big)
\oplus
\widetilde{T}\Big(\frac{b\beta}{b+\beta},\frac{b+\beta}{b\beta},0\Big).
\end{equation}

\subsection{The case $V = \Ttil \otimes \Tbgd$}
Again referring to Lemma \ref{lemma:cubed}, we determine that $V$
contains the lowest weight vectors $t_2 w_3$, $b t_1 w_2 - t_3 w_3$,
and $t_1 w_3$ if and only if $\gamma=0$. Furthermore, $V$ contains the
highest weight vectors
\begin{equation}
\begin{split}
 &t_1 w_1 - \beta t_2 w_3 - \beta (1-\beta\gamma+\delta) t_3 w_2, \\
 &t_1 w_2 + (1-\gamma-\delta) t_2 w_1 - \beta (1-\beta\gamma-\delta) t_3 w_3, \\
 &t_1 w_3 + (1-\beta\gamma+\delta) t_2 w_1 + (1-\beta\gamma+\delta)
 (1-\beta\gamma-\delta) t_3 w_1
\end{split}
\end{equation}
if and only if
\begin{equation}
1+b\beta (1-\beta\gamma+\delta)
(1-\beta\gamma-\delta)=0.
\end{equation}
Furthermore, similarly to the case of $\mtwo \otimes \Tbcd$, the eigenvalues
of $H$ are $\delta$ and $\delta \pm 1$.

Based on this, we split this case into subcases. Note that when we
have both that $\gamma=0$ and $\delta=\pm 1$, it is impossible to have
highest weight vectors. In addition, if $\gamma \neq 0$, then we need
not consider the cases $\delta= \pm 1$ separately from $\delta=0$,
since by Lemma \ref{lemma:dint}, $T(\beta,\gamma,\pm 1)$ is isomorphic
to a module $T(\beta',\gamma',0)$ for some $\beta'$, $\gamma'$.

\subsubsection{The subcase $\gamma=0$; $\delta=0$; $\beta=-\frac{1}{b}$}
In this case, we begin by acting on the three available lowest weight
vectors by $X_+$, followed by acting on the highest weight vectors by
$X_-$. This immediately allows us to compute two indecomposable
submodules of the forms $\mthree$ and $\mone \to \mtwo \gets \mone$.

The above submodules have dimension seven, while $V$ has dimension
nine. The weight vectors linearly independent from the above
submodules have weights 1 and $-1$. Since there are no more highest or
lowest weight vectors, these vectors form an irreducible module of the
form $\mtwo$ only after passage to the quotient. A quick check of the
minimal equation of $X_+ X_-$ acting on the space of vectors of weight
1 shows that $X_+ X_-$ has only two eigenvectors on this space. So we
start with an arbitrary vector that is linearly independent from the
submodules of the last paragraph and find that the complete
decomposition is of the form
\begin{equation}
\mthree \oplus (\mtwo \c+ (\mone \to
\mtwo \gets \mone)).
\end{equation}

\subsubsection{The subcase $\gamma=0$; $\delta=0$; $\beta \neq -\frac{1}{b}$}
Here, we have three lowest weight vectors and no highest weight
vectors. Acting on the lowest weight vectors yields two submodules,
$\widetilde{T}(\frac{b}{1+b\beta},\frac{1+b\beta}{b},0)$ and
$T(\frac{1+b\beta}{b},0,0)$, which contain al of the lowest weight
vectors.

One quickly checks that $\xpm$ does not have a basis of eigenvectors
for the space of weight 1 vectors. Therefore, the subspace of
remaining vectors will be irreducible only upon quotienting. A direct
computation shows that the quotient is of the form
$\widetilde{T}(\frac{b}{1+b\beta},\frac{1+b\beta}{b},0)$, giving the
complete decomposition
\begin{equation}
\widetilde{T}\Big(\frac{b}{1+b\beta},\frac{1+b\beta}{b},0\Big) \c+
\Big(\widetilde{T}\Big(\frac{b}{1+b\beta},\frac{1+b\beta}{b},0\Big)
\oplus T\Big(\frac{1+b\beta}{b},0,0\Big)\Big).
\end{equation}

\subsubsection{The subcase $\gamma=0$; $\delta=1$}\label{sec:tildebcd01}
In this case there are three lowest weight vectors but, as noted
above, no highest weight vectors. By acting on the lowest weight
vectors with $X_+$, we immediately get two submodules,
$\widetilde{T}(b,\frac{1}{b},0)$ and $T(\frac{1}{b},0,0)$.

The above submodules contain all of the highest and lowest weight vectors, and we can
check that $X_+ X_-$ does not have a basis of eigenvectors for any
weight space (actually, checking one particular weight space
suffices). Hence, starting with an arbitrary weight vector that is
linearly independent from the two submodules above, we compute the
final submodule, which is irreducible upon quotienting and gives the
decomposition
\begin{equation}
\widetilde{T}\Big(b,\frac{1}{b},0\Big) \c+
\Big(\widetilde{T}\Big(b,\frac{1}{b},0\Big) \oplus
T\Big(\frac{1}{b},0,0\Big)\Big).
\end{equation}

\subsubsection{The subcase $\gamma=0$; $\delta=2$}
This case is completely analogous to that of Section
\ref{sec:tildebcd01}. The end result is also identical.

\subsubsection{The subcase $\gamma=0$; $\delta\neq 0,1$ or $2$; $b\beta
  (1-\delta^2) = -1$} 
For this case, please refer to Section \ref{sec:tildebcd-1}
below. The calculation there is also completely valid for the case
$\gamma=0$.

\subsubsection{The subcase $\gamma=0$; $\delta\neq 0,1$ or $2$; $b\beta
  (1-\delta^2) \neq -1$}
Here, there are no highest weight vectors, but acting on the three
lowest weight vectors by $X_+$ gives the entire decomposition,
\begin{equation}
T(1+b\beta (1-\delta^2),0,\delta-1) \oplus T(1+b\beta
(1-\delta^2),0,\delta) \oplus T(1+b\beta (1-\delta^2),0,\delta+1).
\end{equation}

\subsubsection{The subcase $\gamma \neq 0$; $\delta=0$; $b\beta
  (1-\beta\gamma)^2=-1$}
Here there are no lowest weight vectors, but three highest weight
vectors. Acting on these by $X_-$ yields two 3\defis dimensional
submodules, $T(0,\gamma,0)$ and $T(0,\gamma,-1)$.

A quick check then shows that $X_+ X_-$ does not have a basis of
eigenvectors for any weight space. However, by selecting a weight vector 
linearly independent from the two submodules above and applying
$X_-$ to it, we find that the quotient of $V$ by the two submodules 
above is of the form $T(0,\gamma,-1)$. In all, we get
\begin{equation}
T(0,\gamma,-1) \c+ (T(0,\gamma,0) \otimes T(0,\gamma,-1))
\end{equation}

\subsubsection{The subcase $\gamma \neq 0$; $\delta=0$; $b\beta
  (1-\beta\gamma)^2 \neq -1$}

This case is covered by the calculation of Section
\ref{sec:bc0bcdNothing} below. Note that in the current case, the
condition \eqref{eq:bc0bcdcond} cannot hold. The assumption $\delta=0$
implies $(\delta (\delta+1) (\delta-1))^2 = 0$ and we have assumed
that
\begin{equation}
1 + b
\beta (1 - \beta \gamma + \delta) (1 - \beta \gamma - \delta) \neq
0.
\end{equation}

\subsubsection{The subcase $\gamma \neq 0$; $\delta \neq 0,1$ or $2$; $b\beta
  (1-\beta\gamma+\delta) (1-\beta\gamma-\delta)=-1$}\label{sec:tildebcd-1}

In this case, there are no lowest weight vectors, but three highest
weight ones. Acting on them by $X_-$ yields a full decomposition,
\begin{equation}
T(0,\gamma,\delta-1) \oplus T (0,\gamma,\delta) \oplus T
(0,\gamma,\delta+1).
\end{equation}

\subsubsection{The subcase $\gamma \neq 0$; $\delta \neq 0,1$ or $2$; $b\beta
  (1-\beta\gamma+\delta) (1-\beta\gamma-\delta) \neq -1$}\label{sec:bc0bcdNothing}

Here, we have neither highest nor lowest weight vectors to
exploit. Hence, we must rely on Lemmas \ref{lemma:cubed} and
\ref{lemma:decomp} and examine the eigenspaces of $X_+ X_-$.

We begin with the space
$V_{\delta+1}$. A straightforward computation shows that the
characteristic polynomial of $X_+ X_-$ acting on $V_{\delta+1}$ is
\begin{equation}\label{eq:bc0charpol}
\lambda^3 + (1-\delta^2) \lambda^2 + \lambda - \frac{\gamma}{b}
(1+b\beta (1-\beta\gamma+\delta) (1-\beta\gamma-\delta)).
\end{equation}

To apply Lemma \ref{lemma:decomp}, we must know how many linearly
independent eigenvectors $\xpm$ has in $V_{\delta+1}$; that is, we
must know what the \emph{minimal} polynomial of $\xpm$ is.

Noting that, in a field of characteristic 3, $(\lambda-\mu)^3 =
\lambda^3 - \mu^3$, it is easily seen that \eqref{eq:bc0charpol}
cannot be written as a perfect cube. So it must have at least two
solutions.

We then note that
\begin{equation}\label{eq:bc0mult2}
(\lambda-\mu)^2 (\lambda-\rho) = \lambda^3 + (\mu-\rho) \lambda^2 +
(\mu^2 - \mu\rho) \lambda - \mu^2 \rho.
\end{equation}
By equating coefficients of $\lambda$, we deduce that
\eqref{eq:bc0charpol} and \eqref{eq:bc0mult2} can be equal if and
only if $\mu = 1-\delta^2$ and $\rho = -\delta^2$, \emph{as well as} the following relation between the
parameters of $\Ttil$ and $\Tbgd$ is satisfied:
\begin{equation}\label{eq:bc0bcdcond}
\frac{\gamma}{b}(1+b\beta (1-\beta\gamma+\delta) (1-\beta\gamma-\delta) = \mu^2 \rho =
-\delta^2 (1-\delta)^2 (1+\delta)^2.
\end{equation}
Of course, even if the characteristic polynomial has a root of
multiplicity two, $\xpm$ may still have a basis of eigenvectors if its
minimal polynomial factors into distinct linear factors, that is if $\xpm$
satisfies the equation
\begin{equation}
(\xpm - (1-\delta^2 I)) (\xpm + \delta^2
I)=0,
\end{equation}
where $I$ is the identity operator on $V_{\delta+1}$. However, a direct
computation shows that this is never the case.

This argumentation shows that we have a direct sum decomposition,
\begin{equation}
T(b_1,c_1,d_1) \oplus T(b_2,c_2,d_2) \oplus T(b_3,c_3,d_3),
\end{equation}
if \eqref{eq:bc0bcdcond} does not hold. If \eqref{eq:bc0bcdcond} holds,
we have a decomposition involving a semidirect sum,
\begin{equation}
T(b_3,c_3,d_3) \c+
(T(b_1,c_1,d_1) \oplus T(b_2,c_2,d_2)).
\end{equation}

Now we are interested in determining the possible values of the
parameters $b_i$, $c_i$, $d_i$. Let $\rho_1$, $\rho_2$, and $\rho_3$
be the (not necessarily distinct) roots of the polynomial
\eqref{eq:bc0charpol}, and assume that at least $\rho_1 \neq
\rho_2$. Now, let $v_{3,i}$ be the distinct eigenvectors of $\xpm$ in
$V_{\delta+1}$, so that $i$ ranges from one to the number of distinct
eigenvectors. Finally, set $v_{2,i} = X_+ v_{3,i}$ and $v_{1,i} =
X_+^2 v_{3,i}$. We take $\left\{v_{1,i}, v_{2,i}, v_{3,i}\right\}$ as
our basis for $T(b_i,c_i,d_i)$ in the matrix representation
\eqref{eq:T}.

With this notation, and recalling Lemma \ref{lemma:params}, we can
determine the parameters. We already know that $X_-^3 v_{j,i} = \gamma
v_{j,i}$ for all $i$ and $j$, so $c_i = \gamma$ for all $i$. Now, on the
one hand, $\xpm v_{3,i} = b_i c_i v_{3,i}$, and the other hand, we
know that $\xpm v_{3,i} = \rho_i v_{3,i}$. Therefore, $b_i =
\frac{\rho_i}{\gamma}$. Finally, since each $v_{3,i}$ belongs to
$V_{\delta+1}$, we know from the algebra equations for $\sltwo$ that
$v_{2,i}$ is of weight $(\delta+1)+2 = \delta$. So $d_i =
\delta$. This determines the parameters completely for the case that
$\xpm$ has a basis of eigenvectors for $V_{\delta+1}$.

If $\xpm$ has only two distinct eigenvectors for $V_{\delta+1}$, $b_3$
is yet to be determined. To do this, we note the following general
fact. Let $A: W \to W$ be a linear endomorphism of a
finite\defis dimensional vector space $W$, and let $U$ be an $A$\defis invariant
subspace of $W$. Furthermore, let the minimal polynomial of $A: W \to
W$ be $m$, that of $A: U \to U$ be $m_1$, and that of the induced map
$A: W/U \to W/U$ be $m_2$. Then $m = m_1 \cdot m_2$.

With this in mind, assume we have nondistinct eigenvalues $\rho_1 =
\rho_3 = 1-\delta^2$ and $\rho_2 = -\delta^2$. Then the minimal
polynomial of $\xpm$ on $V_{\delta+1}$ is
\begin{equation}
m = (\lambda - (1-\delta^2))^2
(\lambda + \delta^2).
\end{equation}
Using the notation of the previous paragraph with
$A = \xpm$, $W = V_{\delta+1}$, and $U = \mathrm{span}(v_{3,1},v_{3,2})$, we
then have
\begin{equation}
m_1 = (\lambda - (1-\delta^2)) (\lambda + \delta^2),
\end{equation}
implying $m_2 = (\lambda - (1-\delta^2))$. Hence, letting $v_{3,3}$ be
a representative for a nonzero vector in the quotient space $W/U$, we get
$\xpm$ is $\xpm v_{3,3} = (1-\delta^2) v_{3,3}$. Therefore,
\begin{equation}
b_3 = \frac{1-\delta^2}{c}.
\end{equation}

Thus, we have determined all parameters for the decomposition.

\subsection{The case $\Tbcd \otimes \Tbgd$}

By Lemma \ref{lemma:cubed}, we have lowest weight vectors if and only
if $c+\gamma = 0$, and highest weight vectors if and only if $b a_1
a_2 + \beta \alpha_1 \alpha_2 = 0$, where
\begin{equation}
\begin{aligned}
a_1 = bc+d-1, & \quad & \alpha_1 = \beta\gamma+\delta-1, \\
a_2 = bc-d-1, & \quad & \alpha_2 = \beta\gamma-\delta-1.
\end{aligned}
\end{equation}

The lowest weight vectors are then
\begin{equation}
\begin{split}
 &c v_1 w_1 + v_2 w_3 - v_3 w_2, \\
 &c v_1 w_2 - c v_2 w_1 - v_3 w_3, \\
 &v_1 w_3 - v_2 w_2 + v_3 w_1,
\end{split}
\end{equation}
and the highest weight vectors are
\begin{equation}
\begin{split}
 &a_1 a_2 v_1 w_1 - \beta a_2 v_2 w_3 + \beta \alpha_2 v_3 w_2, \\
 &a_1 a_2 v_1 w_2 - a_2 \alpha_1 v_2 w_1 + \beta \alpha_1 v_3 w_3, \\
 &a_1 a_2 v_1 w_3 - a_2 \alpha_2 v_2 w_2 + \alpha_1 \alpha_2 v_3 w_1.
\end{split}
\end{equation}

The highest weight vectors were computed using the following method
(here, as an example, for the weight space $\mathrm{span} (v_1 w_1,
v_2 w_3, v_3 w_2)$). We have, for the action of $X_+$ on this weight
space,
\begin{equation}
\begin{split}
 &X_+ (v_1 w_1) = \beta v_1 w_3 + b v_3 w_1,\\
 & X_+ (v_2 w_3) = a_1 v_1 w_3 + \alpha_2 v_2 w_2, \\
 &X_+ (v_3 w_2) = a_2 v_2 w_2 + \alpha_1 v_3 w_1.
\end{split}
\end{equation}
We first try to cancel the factors of $v_1 w_3$, noting that
\begin{equation}
X_+ (a_1 v_1 w_1 - \beta v_1 w_3) = -\beta \alpha_2 v_2 w_2 + b a_1 v_3 w_1.
\end{equation}
From here, we cancel the factors of $v_2 w_2$:
\begin{equation}
X_+(a_2 (a_1 v_1 w_1 - \beta v_2 w_3) + \beta \alpha_2 v_3 w_2) = (b
a_1 a_2 + \beta \alpha_1 \alpha_2) v_3 w_1
\end{equation}
Since we have assumed that $b a_1 a_2 + \beta \alpha_1 \alpha_2 = 0$,
this tells us that $a_1 a_2 v_1 w_1 - \beta a_2 v_2 w_3 + \beta
\alpha_2 v_3 w_2$ is a highest weight vector.

The problem with this method of computing the highest weight vector is
that the vector we end up with might be the zero vector. We can try to
rectify this by changing the order of the factors that we
cancel. Doing so gives us two additional ``representations'' (by an
abuse of language) for the highest weight vectors:
\begin{equation}
\begin{split}
 &a_1 \alpha_1 v_1 w_1 - \beta \alpha_1 v_2 w_3 - b a_1 v_3 w_2, \\
 &\beta \alpha_2 v_1 w_2 + b a_2 v_2 w_1 - b \beta v_3 w_3, \\
 &\beta \alpha_1 v_1 w_3 + b a_2 v_2 w_2 + b \alpha_1 v_3 w_1
\end{split}
\end{equation}
and
\begin{equation}
\begin{split}
 &\alpha_1 \alpha_2 v_1 w_1 + b a_2 v_2 w_3 - b \alpha_2 v_3 w_2, \\
 &a_1 \alpha_2 v_1 w_2 - \alpha_1 \alpha_2 v_2 w_1 - b a_1 v_3 w_3, \\
 &\beta a_1 v_1 w_3 - \beta \alpha_2 v_2 w_2 + b a_1 v_3 w_1.
\end{split}
\end{equation}

We then hope that one of these ``representations'' is nonzero. However,
for certain choices of the parameters, all three ``representations''
of some highest weight vector are zero. A lengthy but straightforward
analysis shows that the only such choices are given by Table
\ref{table:highestzero}.

\renewcommand*{\arraystretch}{1.2}

\begin{table}
\begin{center}
\begin{tabular}{c|c|c|c}
$\boldsymbol{b}$ & $\boldsymbol{d}$ & $\boldsymbol{\beta}$ &
$\boldsymbol{\delta}$ \\
\hline
0 & 1 & 0 & 1 \\
0 & $-1$ & 0 & $-1$ \\
0 & 1 & $\frac{1}{\gamma}$ & 0 \\
0 & $-1$ & 0 & 1 \\
0 & 1 & 0 & $-1$ \\
0 & $-1$ & $\frac{1}{\gamma}$ & 0 \\
$\frac{1}{c}$ & 0 & 0 & 1 \\
$\frac{1}{c}$ & 0 & 0 & $-1$ \\
$\frac{1}{c}$ & 0 & $\frac{1}{\gamma}$ & 0 \\
\end{tabular}
\caption{Parameter values for which all ``representations'' of some
  highest weight vector are zero}\label{table:highestzero}
\end{center}
\end{table}

We do not have to consider all of these cases separately,
however. Note that in all of these cases, $c \neq 0$. Whenever
$b=\frac{1}{c}$, it is assumed, and whenever $b=0$, we have $d=\pm 1$,
and so we must have $c \neq 0$ if $T(b,c,d)$ is to be irreducible (see
the Introduction). Likewise, $\gamma \neq 0$ in all of these
cases. Therefore, we may use Lemma \ref{lemma:dint} (and the explicit
calculations of its proof) to see that $T(0,c,\pm 1) \simeq
T(\frac{1}{c},c,0)$ and $T(0,\gamma,\pm 1) \simeq
T(\frac{1}{\gamma},\gamma,0)$. Hence, all nine cases are equivalent to
the case where $b=\frac{1}{c}$, $d=0$, $\beta=\frac{1}{\gamma}$, and
$\delta=0$. So this means we must consider that particular case
separately from the other cases where highest weight vectors are
present.

One may ask the related question of whether there is only one highest
weight vector in each weight space. For the lowest weight vectors this
is clear by inspection, but in the case of highest weight vectors, it
is not so easily seen directly whether this is true. We may, however,
consider the characteristic polynomials of $\xpm$ on the weight
spaces. These are
\begin{align}
 &\lambda^3 + \lambda^2 + ((d+\delta)^2 - 1) \lambda + (c+\gamma) (b a_1
a_2 \beta \alpha_1 \alpha_2)\text{ on $V_{d+\delta+1}$,}\\
 &\lambda^3 + \lambda^2 + (d+\delta) ((d+\delta) - 1) \lambda +
 (c+\gamma) (b a_1 a_2 \beta \alpha_1 \alpha_2)\text{ on
 $V_{d+\delta-1}$,}\\
 &\lambda^3 + \lambda^2 + (d+\delta) ((d+\delta) + 1) \lambda + (c+\gamma) (b a_1
 a_2 \beta \alpha_1 \alpha_2)\text{ on $V_{d+\delta}$.}
\end{align}

From this, we clearly see that $\xpm$ can have at most two
eigenvectors with eigenvalue~0, since the geometric multiplicity of
an eigenvalue is at most its algebraic multiplicity in the
characteristic equation. Furthermore, we note that if $c+\gamma=0$,
one of these zero eigenvectors will come from a vector $v$ with $X_- v =
0$. Therefore, we focus on the case where $c+\gamma \neq 0$. Another
lengthy but straightforward calculation shows us that the only cases
where $\xpm$ has two eigenvectors with eigenvalue 0 in some weight
space are again exactly given by Table \ref{table:highestzero}. Hence,
we have no extra cases here to consider specially.

For the action of $H$, we note that the possible weights for
vectors in $V$ are $d+\delta$ and $d+\delta \pm 1$. As in the case of
$\Ttil \otimes \Tbgd$, we will be concerned with when these weights
can be 0, 1, or $-1$, since in these cases we will see phenomena that
are not possible otherwise. Hence, we consider $d+\delta=0$ or $\pm 1$
separately.

However, for $\gamma \neq -c$ (i.e., when there are no lowest weight
vectors), we may use Lemma \ref{lemma:dint} to reduce the cases
$d+\delta=\pm 1$ to the case $d+\delta=0$. This is because we can
assume that either $c$ or $\gamma$ is nonzero. Without loss of
generality, we assume $c \neq 0$. We can then use the isomorphisms
\begin{equation}
T(b,c,0) \simeq T(b',c',1) \simeq T(b'',c'',-1)
\end{equation}
(for appropriate $b'$, $b''$, $c'$ and $c''$) to modify the value of
$d+\delta$.

If $\gamma = -c$, on the other hand, we could have $\gamma = 0 = c$,
in which case Lemma \ref{lemma:dint} is inapplicable. For the
calculations where $d+\delta=\pm 1$, we therefore assume $\gamma = 0 =
c$. This simplifies equations and sometimes allows sharper
decompositions.

Finally, we note the following factorizations for $K = b a_1 a_2 + \beta
\alpha_1 \alpha_2$ for special values of $d+\delta$ when
$\gamma=-c$:
\begin{equation}
\begin{split}
 &d+\delta = 0 \Longrightarrow K = (b+\beta) (1 + bc + b^2 c^2 - d^2 - c\beta - b c^2
 \beta + c^2 \beta), \\
 &d+\delta = 1 \Longrightarrow K = (1+d) ((1-d) b - d\beta), \\
 &d+\delta = -1 \Longrightarrow K = (1-d) ((1+d) b + d\beta).
\end{split}
\end{equation}

With all of this in mind, we are now ready to break down the necessary
subcases. Except in the two subcases where we explicitly state this
to be the case, we assume that we do \emph{not} have all of the
conditions $b=\frac{1}{c}$, $d=0$, $\beta = \frac{1}{\gamma}$, and $\delta=0$. (The same
assumption goes for Tables \ref{table:bcdTimesbcdLowest} and \ref{table:bcdTimesbcd}.)

\subsubsection{The subcase $\gamma=-c$; $b=\frac{1}{c}$; $d=0$; $\beta =
  \frac{1}{\gamma}$; $\delta=0$}

This computation is implied by that in \ref{sec:bcdbcdExceptional}
below. We note that $T(0,0,0) \simeq \mthree$, $T(0,0,1) \simeq \mone
\to \mtwo$, and $T(0,0,-1) \simeq \mtwo \to \mone$.

\subsubsection{The subcase $\gamma=-c$; $d+\delta=0$; $\beta=-b$; $1 +
  bc + b^2 c^2 - d^2 - c\beta - b c^2 \beta + c^2 \beta^2 = 0$}\label{sec:bcdbcd-c01}

In this case, we have three lowest weight vectors. Acting on them by
$X_+$, we obtain three submodules, of the forms $\mone$, $\mtwo$, and
$\mthree$. This exhausts both the highest and lowest weight
vectors.

Examining the action of $\xpm$ on $V_1$, we notice by calculating the
characteristic and minimal polynomials that $\xpm$ has two
eigenvectors with eigenvalue 1. We have only exploited one of these;
acting on the other by $X_+$ and $X_-$ gives a module that is of the
form $\mtwo$ after quotienting with the $\mone$ from the last
paragraph. Together, these submodules have dimension eight, while V
has dimension nine, so we deduce the complete decomposition
\begin{equation}
\mone \c+ ((\mtwo \c+ \mone) \oplus \mtwo \oplus \mthree).
\end{equation}

\subsubsection{The subcase $\gamma=-c$; $d+\delta=0$; $\beta=-b$; $1 +
  bc + b^2 c^2 - d^2 - c\beta - b c^2 \beta + c^2 \beta^2 \neq 0$}\label{sec:bcdbcd-c02}

As in the last case, we begin with acting by $X_+$ on the lowest weight vectors,
which gives two submodules of the forms $\mthree$ and $\mone \gets
\mtwo$. These submodules contain all but one highest weight
vector, which has weight 1. Acting on it by $X_-$ gives a submodule of
the form $\mtwo$. Since there are no more highest or lowest weight
vectors, the complete decomposition is of the form
\begin{equation}
\mone \c+ (\mtwo \c+ (\mthree \oplus (\mone \gets \mtwo))).
\end{equation}

\subsubsection{The subcase $\gamma=-c$; $d+\delta=0$; $\beta \neq -b$; $1 +
  bc + b^2 c^2 - d^2 - c\beta - b c^2 \beta + c^2 \beta^2 = 0$}\label{sec:bcdbcd-c03}

This case proceeds completely analogously to above, only after
exploiting the lowest weight vectors, we have two submodules of the
forms $\mthree$ and $\mtwo \gets \mone$. The one remaining highest
weight vector is of weight 0. Acting on it by $X_-$ yields a vector in
$\mthree \oplus (\mtwo \gets \mone)$. With no remaining highest or
lowest weight vectors, the decomposition is of the form
\begin{equation}
\mtwo \c+ (\mone \c+ (\mthree \oplus (\mtwo \gets \mone))).
\end{equation}

\subsubsection{The subcase $\gamma=-c$; $d+\delta=0$; $\beta \neq -b$; $1 +
  bc + b^2 c^2 - d^2 - c\beta - b c^2 \beta + c^2 \beta^2 \neq 0$}\label{sec:bcdbcd-c04}

In this case there are no highest weight vectors. Acting on the lowest
weight vectors by $X_+$ gives two irreducible submodules, $T(K,0,0)$
and $\widetilde{T} (\frac{1}{K},K,0)$. This exhausts all of the lowest
weight vectors. By Lemma \ref{lemma:cubed} there can be no module of
the form $T(b,c,d)$ without highest or lowest weight
vectors. Therefore, the remainder of the module can be irreducible
only upon quotienting. Selecting an arbitrary vector from $V/(T(K,0,0)
\oplus \widetilde{T} (\frac{1}{K},K,0))$ and acting on it by $X_+$
shows that the quotient is of the form $\widetilde{T}
(\frac{1}{K},K,0)$. In all, we get
\begin{equation}
\widetilde{T} \Big(\frac{1}{K},K,0\Big) \c+ \Big(T(K,0,0) \oplus
\widetilde{T} \Big(\frac{1}{K},K,0\Big)\Big)
\end{equation}

\subsubsection{The subcase $\gamma=-c$; $d+\delta=1$; $d=-1$; $(1-d) b
  = d\beta$}\label{sec:bcdbcd-c11}

This case is done completely analogously to Section
\ref{sec:bcdbcd-c01} and gives the same result.

\subsubsection{The subcase $\gamma=-c$; $d+\delta=1$; $d=-1$; $(1-d) b
  \neq d\beta$}\label{sec:bcdbcd-c12}

This case proceeds as in Section \ref{sec:bcdbcd-c02}. However, here the
equations are a bit simpler, and we can achieve the somewhat sharper
decomposition
\begin{equation}
\mone \c+ (\mthree \oplus (\mtwo \to \mone \gets \mtwo)).
\end{equation}

\subsubsection{The subcase $\ref{sec:bcdbcd-c02}\gamma=-c$; $d+\delta=1$; $d \neq -1$; $(1-d) b
  = d\beta$}\label{sec:bcdbcd-c13}

This case proceeds as in Section \ref{sec:bcdbcd-c03}. As in the previous
  case, we can achieve a somewhat sharper decomposition,
\begin{equation}
\mtwo \c+ (\mthree \oplus (\mone \to \mtwo \gets
  \mone)).
\end{equation}

\subsubsection{The subcase $\gamma=-c$; $d+\delta=1$; $d \neq -1$; $(1-d) b
  \neq d\beta$}\label{sec:bcdbcd-c14}

This case is completely analogous to that of Section \ref{sec:bcdbcd-c04}, and
gives an identical result.

\subsubsection{The subcase $\gamma=-c$; $d+\delta=2$; $d = 1$; $(1+d) b
  = -d \beta$}

This case is handled just like Section \ref{sec:bcdbcd-c11} and
gives the same result.

\subsubsection{The subcase $\gamma=-c$; $d+\delta=2$; $d = 1$; $(1+d) b
  \neq -d \beta$}

Here, we proceed as in Section \ref{sec:bcdbcd-c12} and get the
same result.

\subsubsection{The subcase $\gamma=-c$; $d+\delta=2$; $d \neq 1$; $(1+d) b
  = -d \beta$}

This case is analogous to Section \ref{sec:bcdbcd-c13} and again
gives the same result.

\subsubsection{The subcase $\gamma=-c$; $d+\delta=2$; $d \neq 1$; $(1+d) b
  \neq -d \beta$}

Here, we compute the same result as in Section \ref{sec:bcdbcd-c14}
in exactly the same manner.

\subsubsection{The subcase $\gamma=-c$; $d+\delta \neq 0,1$ or $2$; $b a_1
  a_2 = -\beta \alpha_1 \alpha_2$}

The decomposition for this case is implied by that in Section
\ref{sec;bcdbcd-c} simply by substituting $K=0$. The computations are
all still valid.

\subsubsection{The subcase $\gamma=-c$; $d+\delta \neq 0,1$ or $2$; $b a_1
  a_2 \neq -\beta \alpha_1 \alpha_2$}\label{sec;bcdbcd-c}

Here, we have no highest weight vectors, but there are three lowest
weight vectors. Acting on them by $X_+$ immediately gives us the
complete decomposition,
\begin{equation}
T(K,0,d+\delta-1) \oplus T(K,0,d+\delta)
\oplus T(K,0,d+\delta+1).
\end{equation}

\subsubsection{The subcase $\gamma \neq -c$; $b=\frac{1}{c}$; $d=0$; $\beta =
  \frac{1}{\gamma}$; $\delta=0$}\label{sec:bcdbcdExceptional}

In this case, there are no lowest weight vectors, but we have five
highest weight vectors to exploit. Acting on them by $X_-$ gives the
complete decomposition
\begin{equation}
T\Big(0,\frac{b+\beta}{b\beta},-1\Big) \oplus
T\Big(0,\frac{b+\beta}{b\beta},0\Big) \oplus
T\Big(0,\frac{b+\beta}{b\beta},1\Big).
\end{equation}

\subsubsection{The subcase $\gamma \neq -c$; $b a_1 a_2 = - \beta
  \alpha_1 \alpha_2$; $d+\delta=0$}

In this case, there are three highest weight vectors. Acting on them
by $X_-$, we obtain two submodules, $T(0,c+\gamma,-1)$ and
$T(0,c+\gamma,0)$. This exhausts all highest weight vectors.

The two remaining possibilities are that we have some module of the
form $T(b',c',d')$ without highest or lowest weight vectors, or that
the remainder of the module forms an irreducible module only after
quotienting. Lemma \ref{lemma:cubed} rules out the first possibility,
so we must have the second. Furthermore, again by Lemma
\ref{lemma:cubed}, we know that the quotient module must have
the form $T(b',c',d')$ after quotienting, since it can have no lowest
weight vectors. Furthermore, it must have $c'=c+\gamma$.

Let us examine the action of $\xpm$ on $V_0$,
and argue analogously to Section \ref{sec:bc0bcdNothing}. The minimal
polynomial of $\xpm$ on this space is $\lambda^2 (\lambda+1)$. There
is one eigenvector of eigenvalue 0 and one of eigenvalue $-1$. The
quotient of $V_0$ by the span of these eigenvectors is a
1\defis dimensional space, and the minimal polynomial of $\xpm$ on this
space is $\lambda$. Therefore, choosing a basis vector $v'_3$ for the
quotient, we have $\xpm v'_3 = 0$ for the action on the
quotient. Setting $v'_2 = X_+ v'_3$ and $v'_1 = X_+^2 v'_3$ and taking
$\{ v'_1, v'_2, v'_3 \}$ as a basis for the quotient space, we
determine that $b' = 0$ and $d' = 0+2 = -1$.

\subsubsection{The subcase $\gamma \neq -c$; $b a_1 a_2 = - \beta
  \alpha_1 \alpha_2$; $d+\delta \neq 0,1,2$}

Here, we have three highest weight vectors. Acting on them by $X_-$
yields, after a lengthy but straightforward calculation, three
irreducible submodules which form a complete decomposition,
\begin{equation}
T(0,c+\gamma,d+\delta-1) \oplus T(0,c+\gamma,d+\delta) \oplus
T(0,c+\gamma,d+\delta+1).
\end{equation}

\subsubsection{The subcase $\gamma \neq -c$; $b a_1 a_2 \neq - \beta
  \alpha_1 \alpha_2$}\label{sec:bcdbcdNothing}

Here, there are no highest or lowest weight vectors. Therefore, we
want to proceed upon the lines of \ref{sec:bc0bcdNothing},
trying to apply Lemma \ref{lemma:decomp} to determine the
structure of the decomposition.

The characteristic polynomial of $\xpm$ acting on $V_{d+\delta+1}$ is
given by
\begin{equation}\label{eq:bcdcharpol}
\lambda^3 + \lambda^2 + (1-(d+\delta)^2) \lambda - (c+\gamma) (b a_1
a_2 + \beta \alpha_1 \alpha_2).
\end{equation}
We are now interested in using this to determine how many
linearly independent eigenvectors $\xpm$ has on $V_{d+\delta+1}$.

We note, as in \ref{sec:bc0bcdNothing}, that
\eqref{eq:bcdcharpol} must have at least two distinct roots, since
$(\lambda-\mu)^3 = \lambda^3 - \mu^3$ in characteristic 3.

It is, however, possible that \eqref{eq:bcdcharpol} has only two
distinct roots. Equating coefficients of \eqref{eq:bcdcharpol} and
\eqref{eq:bc0mult2} gives that \eqref{eq:bcdcharpol} is of the form
$(\lambda-\mu)^2 (\lambda-\rho)$ if and only if $\mu=-(d+\delta)
(1+d+\delta)$ and $\rho=-1-(d+\delta) (1+d+\delta)$, \emph{as well as}
the following condition on the parameters $b$, $c$, $d$, etc.~is
satisfied:
\begin{equation}\label{eq:bcdcond}
(d+\delta)^2 (1-(d+\delta)^2)^2 = -(c+\gamma) (b a_1 a_2 +\beta
  \alpha_1 \alpha_2).
\end{equation}

We will still have a direct sum decomposition even if
\eqref{eq:bcdcond} is satisfied, provided the minimal polynomial of
$\xpm$ on $V_{d+\delta+1}$ factors into distinct linear factors, i.e.,
\begin{equation}
(\xpm-\mu I) (\xpm-\rho I) = 0,
\end{equation}
This matrix equation can be viewed
as a system of nine equations. Taking sums and differences of these
equations, and using a good deal of brute force, we can reduce these
nine equations to the following three conditions:
\begin{equation}\label{eq:bcdcond2}
\begin{split}
 &bd (c+\gamma) = (d+\delta) (1-(d+\delta)^2), \\
 &\beta \delta = bd, \\
 &\gamma (d-d^3) = c (\delta-\delta^3).
\end{split}
\end{equation}

Therefore, we have the decomposition
\begin{equation}
T(b_3,c_3,d_3)
\c+ (T(b_1,c_1,d_1) \oplus T(b_2,c_2,d_2)).
\end{equation}
if and only if \eqref{eq:bcdcond} is satisfied but \eqref{eq:bcdcond2} is
not. Otherwise, we have the direct sum decomposition
\begin{equation}
T(b_1,c_1,d_1) \oplus
T(b_2,c_2,d_2) \oplus T(b_3,c_3,d_3)
\end{equation}

To determine the $b_i$, $c_i$, and $d_i$, let us repeat our
considerations from Section \ref{sec:bc0bcdNothing} in this case. Let
$\mu_1$, $\mu_2$, and $\mu_3$ be the (not necessarily distinct) roots
of the polynomial \eqref{eq:bcdcharpol}, and assume that at least
$\mu_1 \neq \mu_2$. As before, let $v_{3,i}$ be the distinct
eigenvectors of $\xpm$ in $V_{d+\delta+1}$, so that $i$ ranges from
one to the number of distinct eigenvectors. Then set $v_{2,i} = X_+
v_{3,i}$ and $v_{1,i} = X_+^2 v_{3,i}$. We take $\left\{v_{1,i},
v_{2,i}, v_{3,i}\right\}$ as our basis for $T(b_i,c_i,d_i)$ in the
matrix representation \eqref{eq:T}.

Since $X_-^3 v_{j,i} = (c+\gamma) v_{j,i}$ for all $i$ and $j$, it
follows that $c_i = c+\gamma$ for all $i$. We then have the two
equations $\xpm v_{3,i} = b_i c_i v_{3,i}$ and $\xpm v_{3,i} = \mu_i
v_{3,i}$. These imply that
\begin{equation}
b_i = \frac{\mu_i}{c+\gamma}.
\end{equation}
To determine $d_i$, remember that
each $v_{3,i}$ belongs to $V_{d+\delta+1}$, so $v_{2,i}$ is of weight
$(d+\delta+1)+2 = d+\delta$. Hence, $d_i = d+\delta$.

As in Section \ref{sec:bc0bcdNothing}, we must finally determine $b_3$
in the case that there are only two distinct eigenvectors, that is,
when the minimal polynomial does not factor into linear factors with
multiplicity one. In this case we have $\mu_1 = \mu_3 = -(d+\delta)
(1+d+\delta)$ and $\mu_2 = -1-(d+\delta) (1+d+\delta)$. Here also, The
minimal polynomial of $\xpm$ on $V_{d+\delta+1}$ is
\begin{equation}
(\lambda +
(d+\delta) (1+d+\delta))^2 (\lambda + (1+(d+\delta)
(1+d+\delta))).
\end{equation}
By the same arguments as in Section \ref{sec:bc0bcdNothing}, we deduce that
\begin{equation}
b_3 = \frac{-(d+\delta) (1+d+\delta)}{c+\gamma}.
\end{equation}

This completes our determination of the parameters for this
decomposition, and thus our computations for Theorem \ref{thm:main}.

\renewcommand*{\arraystretch}{1.5}

\begin{landscape}

\begin{longtable}[p]{|l|l|}
\hline
\multicolumn{1}{|c|}{\bf Symbol} & \multicolumn{1}{|c|}{\bf Meaning} \\
\hline
$a_1$ & $bc+d-1$ \\
\hline
$a_2$ & $bc-d-1$ \\
\hline
$\alpha_1$ & $\beta \gamma + \delta -1$ \\
\hline
$\alpha_2$ & $\beta \gamma - \delta -1$ \\
\hline
$J$ & $1 + \alpha_1 \alpha_2 b \beta$ \\
\hline
$K$ & $a_1 a_2 b + \alpha_1 \alpha_2 \beta$ \\
\hline
$D$ & $(d+\delta)^2 (1-(d+\delta)^2)^2$ \\
\hline
$\Delta$ & $(d+\delta) (1+d+\delta)$ \\
\hline
$\rho_1$, $\rho_2$, $\rho_3$ & The roots of the polynomial $\lambda^3 + (1-\delta^2) \lambda^2 + \lambda - \frac{\gamma}{b}
(1+b\beta (1-\beta\gamma+\delta) (1-\beta\gamma-\delta))$. \\
\hline
$\mu_1$, $\mu_2$, $\mu_3$ & The roots of the polynomial $\lambda^3 + \lambda^2 + (1-(d+\delta)^2) \lambda - (c+\gamma) (b a_1
a_2 + \beta \alpha_1 \alpha_2)$. \\
\hline
$M_1$ &
$\xymatrix@dr{
\mtwo \ar[d] \ar[r] & \mone \ar[d] \\
\mone \ar[r] & \mtwo
}$
\\
\hline
\caption{Symbols used in Tables
  \ref{table:TwoTimesbcd}--\ref{table:bcdTimesbcd}}
\end{longtable}

\pagebreak

\begin{longtable}[p]{|l|l|l|l|l|}
\hline
\multicolumn{4}{|c|}{\bf Relations} & \multicolumn{1}{|c|}{\bf
  Decomposition} \\
\hline
$c=0$ & $d=0$ & \multicolumn{2}{|l|}{$b=0$} & $M_1$ \\
\cline{3-5}
 & & \multicolumn{2}{|l|}{$b \neq 0$} & $\widetilde{T}(\frac{1}{b},b,0) \c+ \widetilde{T}(\frac{1}{b},b,0)$\\
\cline{2-5}
 & $d=1$ & \multicolumn{2}{|l|}{$b=0$} & $\mthree \oplus (\mtwo \rightarrow \mone)$ \\
\cline{3-5}
 & & \multicolumn{2}{|l|}{$b \neq 0$} & $\mthree \oplus (\mtwo \c+ \mone)$ \\
\cline{2-5}
 & $d=2$ & \multicolumn{2}{|l|}{$b=0$} & $\mthree \oplus (\mone \rightarrow \mtwo)$ \\
\cline{3-5}
 & & \multicolumn{2}{|l|}{$b \neq 0$} & $\mthree \oplus (\mone \c+ \mtwo)$ \\
\cline{2-5}
 & \multicolumn{3}{|l|}{$d \neq 0,1,2$} & $T(\frac{b(d-1)}{d},0,d-1) \oplus T(\frac{b(d+1)}{d},0,d+1)$ \\
\hline
$c \neq 0$ & $d=0$ & \multicolumn{2}{|l|}{$b=0$} & $T(0,c,1) \c+ T(0,c,1)$ \\
\cline{3-5}
 & & \multicolumn{2}{|l|}{$b=\frac{1}{c}$} & $T(0,c,1) \oplus T(0,c,1)$ \\
\cline{3-5}
 & & \multicolumn{2}{|l|}{$b \neq 0,\frac{1}{c}$} & $T(b+\sqrt{\frac{b}{c}},c,1)
\oplus T(b-\sqrt{\frac{b}{c}},c,1)$\\
\cline{2-5}
 & $d \neq 0,1,2$ & \multicolumn{2}{|l|}{$b=0$} & $T(0,c,d+1) \oplus T(0,c,d-1)$\\
\cline{3-5}
 & & \multicolumn{2}{|l|}{$1-bc+d=0$} & $T(0,c,d) \oplus T(0,c,d+1)$ \\
\cline{3-5}
 & & \multicolumn{2}{|l|}{$1-bc-d=0$} & $T(0,c,d-1) \oplus T(0,c,d)$ \\
\cline{3-5}
 & & $b \neq 0,$ & $bc+d^2 = 0$ & $T(b+\frac{d}{c},c,d+1) \c+T(b+\frac{d}{c},c,d+1)$ \\
\cline{4-5}
 & & $1-bc \pm d \neq 0$ & $bc+d^2 \neq 0$ & $T(b+\frac{d+\sqrt{bc+d^2}}{c},c,d+1) \oplus T(b+\frac{d-\sqrt{bc+d^2}}{c},c,d+1)$ \\
\hline
\caption{$\mtwo \otimes T(b,c,d)$}\label{table:TwoTimesbcd}
\end{longtable}

\pagebreak

\begin{longtable}[p]{|l|l|}
\hline
\multicolumn{1}{|c|}{\bf Relations} & \multicolumn{1}{|c|}{\bf Decomposition} \\
\hline
$b=-\beta$ & $\mthree \oplus (\mtwo \gets \mone) \oplus (\mone \gets \mtwo)$ \\
\hline
$b \neq -\beta$ & $T(\frac{b+\beta}{b\beta},0,0) \oplus
\widetilde{T}(\frac{b\beta}{b+\beta},\frac{b+\beta}{b\beta},0) \oplus \widetilde{T}(\frac{b\beta}{b+\beta},\frac{b+\beta}{b\beta},0)$ \\
\hline
\caption{$\widetilde{T}(b,\frac{1}{b},0) \otimes \widetilde{T}(\beta,\frac{1}{\beta},0)$}
\end{longtable}

\begin{longtable}[p]{|l|l|l|l|l|}
\hline
\multicolumn{4}{|c|}{\bf Relations} & \multicolumn{1}{|c|}{\bf Decomposition} \\
\hline
$\gamma = 0$ & $\delta = 0$ & \multicolumn{2}{|l|}{$\beta = -\frac{1}{b}$} &
$\mthree \oplus (\mtwo \c+ (\mone \to \mtwo \gets \mone))$ \\
\cline{3-5}
 & & \multicolumn{2}{|l|}{$\beta \neq -\frac{1}{b}$} & $\widetilde{T}(\frac{b}{1+b\beta},\frac{1+b\beta}{b},0) \c+ (\widetilde{T}(\frac{b}{1+b\beta},\frac{1+b\beta}{b},0) \oplus T(\frac{1+b\beta}{b},0,0))$ \\
\cline{2-5}
 & \multicolumn{3}{|l|}{$\delta = 1$} & $\widetilde{T}(b,\frac{1}{b},0) \c+
(\widetilde{T}(b,\frac{1}{b},0) \oplus T(\frac{1}{b},0,0))$ \\
\cline{2-5}
 & \multicolumn{3}{|l|}{$\delta = 2$} & $\widetilde{T}(b,\frac{1}{b},0) \c+
(\widetilde{T}(b,\frac{1}{b},0) \oplus
T(\frac{1}{b},0,0))$ \\
\cline{2-5}
 & $\delta \neq 0,1,2$ & \multicolumn{2}{|l|}{$b \beta (1 - \delta^2)
  = -1$} & $T(0,0,\delta-1) \oplus T(0,0,\delta) \oplus T(0,0,\delta+1)$ \\
\cline{3-5}
 & & \multicolumn{2}{|l|}{$b \beta (1 - \delta^2) \neq -1$} & $T(J,0,\delta-1) \oplus T(J,0,\delta) \oplus T(J,0,\delta+1) $ \\
\hline
$\gamma \neq 0$ & $\delta = 0$ & \multicolumn{2}{|l|}{$b \beta (1 -
  \beta \gamma)^2 = -1$} & $T(0,\gamma,-1) \c+ (T(0,\gamma,0) \oplus T(0,\gamma,-1))$ \\
\cline{3-5}
 & & \multicolumn{2}{|l|}{$b \beta (1 - \beta \gamma)^2 \neq -1$} &
$T(\rho_1,c,0) \oplus T(\rho_2,c,0) \oplus T(\rho_3,c,0) $ \\
\cline{2-5}
 & $\delta \neq 0,1,2$ & \multicolumn{2}{|l|}{$J=0$} & $T(0,\gamma,\delta-1) \oplus T(0,\gamma,\delta) \oplus T(0,\gamma,\delta+1)$ \\
\cline{3-5}
 & & $J \neq 0$ & $-\frac{\gamma}{b} J \neq (\delta (\delta+1) (\delta-1))^2$
& $T(\rho_1,c,\delta) \oplus T(\rho_2,c,\delta) \oplus T(\rho_3,c,\delta)$ \\
\cline{4-5}
 & & & $-\frac{\gamma}{b} J = (\delta (\delta+1) (\delta-1))^2$ & $T(\frac{1-\delta^2}{c},c,\delta) \c+ (T(\frac{1-\delta^2}{c},c, \delta) \oplus T(-\frac{\delta^2}{c},c,\delta))$ \\
\hline
\caption{$\widetilde{T}(b,\frac{1}{b},0) \otimes T(\beta,\gamma,\delta)$}
\end{longtable}

\pagebreak

\begin{longtable}[p]{|l|l|l|l|l|}
\hline
\multicolumn{3}{|c|}{\bf Relations} & \multicolumn{1}{|c|}{\bf
  Decomposition} \\
\hline
$d + \delta = 0$ & \multicolumn{2}{|l|}{$b=\frac{1}{c},\ d=0,\ \beta=\frac{1}{\gamma},\ \delta=0$} & $\mthree \oplus (\mone \rightarrow \mtwo) \oplus (\mtwo
\rightarrow \mone)$ \\
\cline{2-4}
 & $\beta = -b$ & $1 + bc + b^2 c^2 - d^2 - c \beta - b c^2 \beta + c^2 \beta^2 = 0$ & $\mone \c+ ((\mtwo \c+
\mone) \oplus \mtwo
\oplus \mthree)$ \\
\cline{3-4}
 & & $1 + bc + b^2 c^2 - d^2 - c \beta - b c^2 \beta + c^2 \beta^2 \neq 0$ & $\mone \c+ (\mtwo \c+ (\mthree \oplus (\mone \leftarrow \mtwo)))$ \\
\cline{2-4}
 & $\beta \neq -b$ & $1 + bc + b^2 c^2 - d^2 - c \beta - b c^2 \beta +c^2 \beta^2 = 0$ & $\mtwo \c+ (\mone \c+
(\mthree \oplus (\mtwo \leftarrow \mone)))$ \\
\cline{3-4}
 & & $1 + bc + b^2 c^2 - d^2 - c \beta - b c^2 \beta + c^2 \beta^2 \neq 0$ & $\widetilde{T}(\frac{1}{K},K,0) \c+
(T(K,0,0) \oplus \widetilde{T}(\frac{1}{K},K,0))$ \\
\hline
 $d + \delta = 1$ & $d=-1$ & $\beta = b$ & $\mone \c+ ((\mtwo \c+
\mone) \oplus \mtwo \oplus \mthree)$ \\
\cline{3-4}
 & & $\beta \neq b$ & $\mone \c+ (\mthree \oplus (\mtwo \rightarrow \mone \leftarrow \mtwo))$ \\
\cline{2-4}
 & $d \neq -1$ & $(1-d)b = d \beta$ & $\mtwo \c+ (\mthree \oplus (\mone \rightarrow \mtwo \leftarrow \mone))$ \\
\cline{3-4}
 & & $(1-d)b \neq d \beta$ & $\widetilde{T}(\frac{1}{K},K,0) \c+ (T(K,0,0) \oplus \widetilde{T}(\frac{1}{K},K,0))$ \\
\hline
 $d + \delta = 2$ & $d=1$ & $\beta = b$ & $\mthree \oplus (\mtwo \c+
\mone) \oplus (\mone \c+ \mtwo)$ \\
\cline{3-4}
 & & $\beta \neq b$ & $\mone \c+ (\mthree \oplus (\mtwo \rightarrow \mone \leftarrow \mtwo))$ \\
\cline{2-4}
 & $d \neq 1$ & $(1+d)b = -d \beta$ & $\mtwo \c+ (\mone \c+ ((\mone \leftarrow \mtwo) \oplus \mthree))$ \\
\cline{3-4}
 & & $(1+d)b \neq -d \beta$ & $\widetilde{T}(\frac{1}{K},K,0) \c+ (T(K,0,0) \oplus \widetilde{T}(\frac{1}{K},K,0))$ \\
\hline
 $d + \delta \neq 0,1,2$ & \multicolumn{2}{|l|}{$b (1-bc+d) (1-bc-d) =
  -\beta (1 + c \beta + \delta) (1 + c \beta - \delta)$} & $T(0,0,d+\delta-1) \oplus T(0,0,d+\delta) \oplus T(0,0,d+\delta+1)$ \\
\cline{2-4}
 & \multicolumn{2}{|l|}{$b (1-bc+d) (1-bc-d)
  \neq -\beta (1 + c \beta + \delta) (1 + c \beta - \delta)$} & $T(K,0,d+\delta-1) \oplus T(K,0,d+\delta) \oplus T(K,0,d+\delta+1)$ \\
\hline
\multicolumn{4}{|l|}{\emph{Note:} Except where explicitly state this
to be the case, we assume that we do \emph{not} have all of the
conditions} \\
\multicolumn{4}{|l|}{$b=\frac{1}{c}$, $d=0$, $\beta = \frac{1}{\gamma}$, and
$\delta=0$.} \\
\hline
\caption{$T(b,c,d) \otimes T(\beta,\gamma,\delta)$ ($\gamma=-c$)}\label{table:bcdTimesbcdLowest}
\end{longtable}

\pagebreak

\begin{longtable}[p]{|l|l|l|l|}
\hline
\multicolumn{3}{|c|}{\bf Relations} & \multicolumn{1}{|c|}{\bf
  Decomposition} \\
\hline
\multicolumn{3}{|l|}{$b=\frac{1}{c},\ d=0,\ \beta =
  \frac{1}{\gamma},\ \delta=0$} & $T(0,\frac{b+\beta}{b\beta},-1) \oplus T(0,\frac{b+\beta}{b\beta},0) \oplus T(0,\frac{b+\beta}{b\beta},1)$ \\
\hline
$K=0$ & \multicolumn{2}{|l|}{$d + \delta = 0$} & $T(0,c+\gamma,-1) \c+ (T(0,c+\gamma,0) \oplus T(0,c+\gamma,-1))$ \\
\cline{2-4}
 & \multicolumn{2}{|l|}{$d + \delta \neq 0,1,2$} & $T(0,c+\gamma,d+\delta-1) \oplus
  T(0,c+\gamma,d+\delta) \oplus T(0,c+\gamma,d+\delta+1)$ \\
\hline
$K \neq 0$ & $D = -K (c+\gamma)$ & $bd(c+\gamma) = \sqrt{D},$ &
$T(\frac{\mu_1}{c+\gamma},c+\gamma,d+\delta) \oplus
T(\frac{\mu_2}{c+\gamma},c+\gamma,d+\delta) \oplus
T(\frac{\mu_3}{c+\gamma},c+\gamma,d+\delta)$ \\
 & & $\beta\delta = bd,$ & \\
 & & $\gamma (d-d^3) = c (\delta-\delta^3)$ & \\
\cline{3-4}
 & & \emph{otherwise} & $T(-\frac{\Delta}{c+\gamma},c+\gamma,d+\delta) \c+ (T(-\frac{\Delta}{c+\gamma},c+\gamma, d+\delta) \oplus T(\frac{-1-\Delta}{c+\gamma}, c+\gamma,d+\delta))$ \\
\cline{2-4}
& \multicolumn{2}{|l|}{$D \neq -K (c+\gamma)$} &
$T(\frac{\mu_1}{c+\gamma},c+\gamma,d+\delta) \oplus
T(\frac{\mu_2}{c+\gamma},c+\gamma,d+\delta) \oplus T(\frac{\mu_3}{c+\gamma},c+\gamma,d+\delta)$ \\
\hline
\multicolumn{4}{|l|}{\emph{Note:} Except where explicitly state this
to be the case, we assume that we do \emph{not} have all of the
conditions} \\
\multicolumn{4}{|l|}{$b=\frac{1}{c}$, $d=0$, $\beta = \frac{1}{\gamma}$, and
$\delta=0$.} \\
\hline
\caption{$T(b,c,d) \otimes T(\beta,\gamma,\delta)$ ($\gamma \neq
  c$)}\label{table:bcdTimesbcd}
\end{longtable}

\end{landscape}

\bibliography{main}
\bibliographystyle{hamsplain}

\end{document}